\documentclass[12pt,cd]{amsart}
\usepackage[notref,notcite]{showkeys}
\usepackage[plainpages,backref,urlcolor=blue]{hyperref}
\usepackage{amsmath}
\usepackage{amsxtra}
\usepackage{amscd}
\usepackage{amsthm}
\usepackage{amsfonts}
\usepackage{amssymb}
\usepackage {pstricks}
\usepackage{pstricks,pst-node}
\usepackage[all]{xy}

\usepackage{comment} 

\newtheorem{theorem}{Theorem}[section]

\newtheorem{lemma}[theorem]{Lemma}

\newtheorem{corollary}[theorem]{Corollary}
\newtheorem{proposition}[theorem]{Proposition}

\theoremstyle{definition}
\newtheorem{definition}[theorem]{Definition}
\newtheorem{example}[theorem]{Example}

\theoremstyle{remark}
\newtheorem{remark}[theorem]{Remark}

\numberwithin{equation}{section}

\newcommand{\be}{\begin{equation}}
\newcommand{\ee}{\end{equation}}

\newcommand{\mc}{\mathcal}

\newcommand{\C}{{\mathbb C}}

\newcommand{\Z}{{\mathbb Z}}

\newcommand{\bU}{{\mathbb U}}

\newcommand{\CB}{{\mathcal B}}

\newcommand{\CF}{{\mathcal F}}
\newcommand{\CP}{{\mathcal P}}

\newcommand{\Ker}{{\rm{Ker}}}
\newcommand{\id}{{\rm{id}}}

\newcommand{\mf}{\mathfrak}
\newcommand{\fg}{{\mf g}}

\newcommand{\cD}{\mc D}
\newcommand{\cE}{\mc E}

\newcommand{\La}{\Lambda}
\newcommand{\la}{\lambda}
\newcommand{\ve}{\varepsilon}

\newcommand{\zt}{{\Z/2\Z}}

\newcommand{\End}{{\rm{End}}}

\newcommand{\Hom}{{\rm{Hom}}}

\newcommand{\GL}{{\rm{GL}}}
\newcommand{\Sym}{{\rm{Sym}}}
\newcommand{\Exp}{{\rm{Exp}}}

\newcommand{\inv}{{^{-1}}}

\advance\headheight by 2pt

\newcommand{\gl}{{\mathfrak {gl}}}

\newcommand{\osp}{{\mathfrak {osp}}}

\newcommand{\ul}{\underline}
\newcommand{\lr}{\longrightarrow}
\newcommand{\ol}{\overline}

\newcommand{\Sp}{{\rm Sp}}

\newcommand{\ot}{\otimes}

\newcommand{\sdim}{{\rm sdim\,}}

\newcommand{\OSp}{{\rm OSp}}

\begin{document}

\normalfont

\title[Invariant theory of the orthosymplectic supergroup]{The second fundamental theorem of
invariant\\ theory for the  orthosymplectic supergroup}

\author{G.I. Lehrer and R.B. Zhang}
\thanks{This research was supported by the Australian Research Council}
\address{School of Mathematics and Statistics,
University of Sydney, N.S.W. 2006, Australia}
\email{gustav.lehrer@sydney.edu.au, ruibin.zhang@sydney.edu.au}
\begin{abstract}
In a previous work we established a super Schur-Weyl-Brauer duality
between the orthosymplectic supergroup of superdimension
$(m|2n)$ and the Brauer algebra with parameter $m-2n$. This led to a
proof of the first fundamental theorem of invariant theory, using some elementary
algebraic supergeometry, and based upon an idea of Atiyah. In this work we use the same circle of ideas to
prove the second fundamental theorem for the orthosymplectic
supergroup. The proof uses algebraic supergeometry to reduce the problem to the
case of the general linear supergroup, which is understood. The main result has a succinct
formulation in terms of Brauer diagrams. Our proof includes new proofs of the
corresponding second fundamental theorems for the classical orthogonal and symplectic
groups, as well as their quantum analogues. These new proofs are independent of the Capelli
identities, which are replaced by algebraic geometric arguments.
\end{abstract}



\maketitle

\tableofcontents

\section{Introduction}\label{s:intro}
This paper is a sequel to \cite{LZ6}, in which we proved the first fundamental theorem of invariant theory
for the orthosymplectic Lie group $G=\OSp(m|2n,\Lambda)$ over the infinite dimensional Grassmann algebra $\Lambda$.
That theorem provides a set of generators for the invariants of  $G$ on $V^{\ot r}$, where $V$ is the `natural' representation
of $G$ on the superspace $V$ of superdimension $(m|2n)$; this theorem is equivalent to the statement that there is a
surjection $B_r(m-2n)\lr \End_G(V^{\ot r})$, where $B_r(m-2n)$ is the $r$-string Brauer algebra with
parameter $m-2n$. In this paper we give a linear description of all relations among these generators,
thus proving the second fundamental theorem of invariant theory.

We shall maintain the notation of {\it op.~cit.}, and use its results.

In this Introduction, we recall the basic notation and general setup of \cite{LZ6}, and in the next section we state
the main result of \cite{LZ6}, as well as giving a brief discussion of the invariant theory of the general linear supergroup $\GL(V)$,
which will figure in the proof of our main result. In the following section we state and prove the result.
Our method is based upon an idea due to Atiyah,  Bott and Patodi \cite[Appendix]{ABP} in the classical case, to use algebraic supergeometry
to reduce the theorem to the case of the general linear supergroup $\GL(V)$. In particular, we use geometric means
to define an injective $\La$-linear map $h:\End_G(V^{\ot r})\lr \left((V^*)^{\ot r}\ot V^{\ot r}\right)^{\GL(V)}$. 
Using the fact that a full description of $\left((V^*)^{\ot r}\ot V^{\ot r}\right)^{\GL(V)}$ 
is available in terms of generators and relations, this provides the means to prove our main result, which has a 
particularly attractive formulation in terms of diagrams in the Brauer category \cite{LZ5}.   

\subsection{Linear superalgebra} (See \cite[\S\S2.1,2.2]{LZ6})
We denote by $V_\C=(V_{\C})_{\bar0}\oplus (V_\C)_{\bar1}$ a $\zt$-graded complex vector space
of superdimension $\sdim V=(m|2n)$, so that $\dim(V_\C)_{\bar0}=m$ and $\dim(V_\C)_{\bar1}=2n$. If $\Lambda(N)$
denotes the exterior algebra on $\C^N$, then the natural inclusion $\C^N\subset \C^{N+1}$ induces an inclusion $\Lambda(N)\subset\Lambda(N+1)$. We write $\Lambda:=\lim\limits_{\to}\Lambda(N)$ for the direct limit; this is the Grassmann algebra.

Since $\Lambda$ is evidently a $\zt$-graded algebra (graded by partity of the degree),
we may form $V=V_\C\ot_\C\Lambda$. This is a $\Z_2$-graded $\Lambda$-module, and we write $V=V_{\bar0}\oplus V_{\bar1}$.
For a homogeneous element $v\in V$, we denote by $[v]$ the parity
of $v$. Thus $[v]=0$ if $v\in V_{\bar0}$ while $[v]=1$ for $v\in V_{\bar1}$.

The $\Lambda$-module $\End_\Lambda(V)$ of $\Lambda$-endomorphisms of $V$ is $\zt$-graded,
and we write $E:=\End_\Lambda(V)_{\bar0}$.

The group $\GL(V)$ is the group of invertible elements of $E$, that is
\[
\GL(V)=\{g\in \End_{\Lambda}(V)_{\bar0}\mid g \ \text{ invertible}\}.
\]
This is very much in the spirit of the physics literaure \cite{SS, DeW}.

\subsection{The orthosymplectic supergroup} Next suppose given a non-degenerate even bilinear form
\[
(-, - )_\C: V_\C\times V_\C \longrightarrow \C,
\]
which is supersymmetric, that is,  $( u  , v )_\C =(-1)^{[u][v]} ( v,  u )_\C$
for all $u, v\in V_\C$.   This implies that the form is
symmetric on $(V_\C)_{\bar0}\times (V_\C)_{\bar0}$ and
skew symmetric on $(V_\C)_{\bar1}\times (V_\C)_{\bar1}$, and
satisfies $(V_{\bar0},  V_{\bar1})_\C=0=(V_{\bar1}, V_{\bar0})_\C$.
Also, by non-degeneracy, $\dim (V_\C)_{\bar1}=2n$ must be even.
We call this a nondegenerate {\em supersymmetric  form}.   Let
$
\eta = \begin{pmatrix} I_m & 0 \\ 0 & J\end{pmatrix},
$
where $I_m$ is the identity matrix of size $m\times m$ and $J$ is a skew symmetric matrix of size
$2n\times 2n$ given by
$J=diag(\sigma, \dots, \sigma)$ with $\sigma=\begin{pmatrix} 0 & -1 \\ 1 & 0\end{pmatrix}$.
Then there exists an ordered homogeneous basis $\cE=(e_1,  e_2, \dots, e_{m+2n}) $ of $V_\C$
such that
\begin{eqnarray}\label{eq:standard}
(e_a, e_b)=\eta_{a b}, \quad \text{for all $a, b$}.
\end{eqnarray}

Let $V:=V_\C\otimes\Lambda$. The given form $(- , - )_\C$ extends uniquely to a $\Lambda$-bilinear form
\be\label{eq:laform}
( - , - ): V\times V\longrightarrow \Lambda,
\ee
which is even and nongenerate,  and is supersymmetric in the sense that
$(v, w)=(-1)^{[v][w]}(w, v)$ for all $v, w\in V$. We call $V$ an
orthosymplectic superspace, and call $\cE$
an {\em orthosymplectic basis} of $V_\C$ and of $V$.
Note that $(-,-)$ is $\Lambda$-bilinear in the sense that
for $\la,\la'\in\La$ and $v,v'\in V$, we have
\be\label{eq:bil}
(\la v, v'\la')=\la(v,v')\la'.
\ee

\begin{definition}\label{def:sg} The orthosymplectic supergroup $G=\OSp(V)$ is defined as the set of elements
of $\GL(V)$ which preserve the form $(-,-)$. Specifically, the orthosymplectic
supergroup $G:=\{g\in\GL(V)\mid (gv,gw)=(v,w)\text{ for all }v,w\in V\}$.
\end{definition}

Although $G=\OSp(V)$ will be the main focus of this work,
we shall also refer to the following related algebras and groups.

We write $\osp(V_\C)$ for the orthosymplectic Lie superalgebra over $\C$ \cite{K, S};
this is the Lie sub-superalgebra of $\gl(V_\C)$  given by
\[
\osp(V_\C) =\{X\in \End_\C(V_\C) \mid (X v, w)_\C + (-1)^{[X][v]}(v, X w)_\C = 0,\  \forall v, w\in V_\C\}.
\]
Then $\osp(V)=(\osp(V_\C)\otimes\Lambda)_{\bar0}$.

Write $\OSp(V)_0={\rm O}((V_\C)_{\bar0})\times{\rm Sp}((V_\C)_{\bar1})$;
this group, which could be thought of as the degree zero part of $\OSp(V)$,
acts on $\osp(V_\C)$ by conjugation, and
both $\OSp(V)_0$ and $\osp(V_\C)$ act  naturally on $V_\C^{\otimes r}$ for all $r$.
Their actions are compatible in the following sense. For all $g\in\OSp(V)_0$,
$X\in \osp(V_\C)$ and $w\in V_\C^{\otimes r}$, we have
$$
g(X w) =Ad_g(X)(g w),
$$
where $Ad_g$ denotes the conjugation action of $g\in\OSp(V)_0$ on $\osp(V_\C)$.

\begin{remark}\label{rem:HC-pair}
Note that $(\OSp(V)_0, \osp(V_\C))$ is a Harish-Chandra super pair
(see \cite{DM} for details).
\end{remark}
\section{Invariant theory for $\GL(V)$}
In this section we take $V_\C$ to be a $\C$-superspace with $\sdim (V)=(m|\ell)$, and write $V=V_\C\ot_\C\Lambda$,
for the corresponding $\Z_2$-graded $\Lambda$-module, where $\Lambda$ is the Grassmann algebra.
Given any two $\Z_2$-graded $\C$-vector spaces,  $V_\C,W_\C$, we may form
the tensor product $V_\C\ot_\C W_\C$ (resp. $V\ot_\Lambda W$). This is $\Z_2$-graded in the usual way.
We have a super-inversion $\tau:V_\C\ot W_\C\lr W_\C\ot V_\C$, given by
$$
\tau(v\ot w)=(-1)^{[v][w]}w\ot v,
$$
for homogeneous $v\in V_\C,w\in W_\C$, and extended linearly.
We shall be particularly interested in this inversion when $V_\C=W_\C$, in which case it defines an endomorphism
of $V_\C^{\ot 2}$.

Writing $V=V_\C\ot_\C\Lambda$ and similarly for $W$, the involution $\tau$ extends uniquely to a $\Lambda$-module homomorphism
\be\label{eq:tau}
\tau:V\ot W\to W\ot V,\;\;\tau(v\ot w)=(-1)^{[v][w]}w\ot v,
\ee
where $v\in V,w\in W$ are homogeneous.

In particular, we can form the tensor powers $T^r(V_\C)=V_\C^{\ot r}=V^{\ot_\C r}$
and $T^r(V)=V^{\ot_\Lambda r}$. These spaces inherit actions from the relevant groups and algebras, and
we shall be concerned with the algebras $\End_{G}(V^{\ot r})$ for various $G$, especially $G=\GL(V)$,
and  $G=\OSp(V)$ when $\ell=2n$ is even.
The first fundamental theorem (FFT) describes generators, and the second fundamental theorem (SFT) describes
relations for these algebras. In our previous paper \cite{LZ6} we proved the FFT for $\OSp(V)$, while the main purpose of this paper
is to prove an SFT for this case.

\subsection{Invariant theory for $\GL(V)$-the first fundamental theorem} We shall require a little background on $\GL(V)$.
The general linear Lie superalgebra $\gl(V_\C)$ is defined as the complex $\Z_2$-graded algebra
$\End_\C(V_\C)$ with a bilinear
Lie bracket defined for $X, Y\in \gl(V_\C)$ by
\[
[X, Y]=X Y - (-1)^{[X][Y]}Y X,
\]
where the right hand side is defined by composition of endomorphisms.
This algebra is often referred to as $\gl(m|\ell)$.

Let $\widetilde\gl(V)=\gl(V_\C)\otimes_\C\Lambda$ and regard it as a Lie superalgebra
over $\Lambda$ with a $\Lambda$-bilinear Lie bracket defined by
\[
[X\otimes\lambda, Y\otimes\mu]=[X, Y]\otimes (-1)^{[\lambda]([Y\otimes\mu])}\mu\lambda,
\]
for all $X, Y\in \gl(V_\C)$ and $\mu, \nu\in\Lambda$. Then $\gl(V)=(\gl(V_\C)\otimes_\C\Lambda)_{\bar0}$
forms a Lie subalgebra of $\widetilde\gl(V)$ over $\Lambda_{\bar 0}$, which will be referred to as the
Lie algebra of $\GL(V)$.
Note that $\gl(V)=\gl(m|\ell)_{\bar0}\ot_\C\Lambda_{\bar0}\oplus \gl(m|\ell)_{\bar1}\ot_\C\Lambda_{\bar1}$.
Thus both $\gl(V)$ and $\GL(V)$ encode the graded structure of $V_\C$ and of $V$ (cf. Lemma \ref{lem:exp} below).
There is a natural $\widetilde\gl(V)$ action on $V$ given by
$(X\otimes\lambda).(v\otimes\mu)=
X.v\otimes (-1)^{[\lambda]([v\otimes\mu])}\mu\lambda$.
It restricts to an action of $\gl(V)$.

The following lemma provides a useful link between
$\gl(m|\ell)$ and $\GL(V)$. Its proof may be found in \cite[Proposition 5.2]{SZ}
and \cite[Lemma 2.6]{LZ6} (see also \cite{Be}).
\begin{lemma}\label{lem:exp}
\begin{enumerate}
\item Given any $X\in \gl(V)$, let
$\exp(X):=\sum_{i=0}^\infty \frac{X^i}{i!}$.
Then $\exp(X)$  is a well defined automorphism of $V$ which lies in $\GL(V)$.
Hence there exists a map
\[
{\rm Exp}: \gl(V) \longrightarrow \GL(V), \quad X\mapsto \exp(X).
\]
\item The image of $\Exp$ generates $\GL(V)$.
\end{enumerate}
\end{lemma}

Now the general linear supergroup $\GL(V)$ acts on $T^r(V)$ by
$
g.w=g w_1\otimes...\otimes g w_r
$
for any $w=w_1\otimes...\otimes w_r$ and $g\in \GL(V)$.
The corresponding $\gl(V)$-action on $V^{\otimes_\Lambda r}$ is defined for all $X\in\gl(V)$ by
\be\label{eq:glvaction}
\begin{aligned}
X.w& = Xw_1\otimes w_2\otimes  \dots \otimes  w_r+ w_1\otimes X w_2\otimes  \dots \otimes  w_r\\
&+\dots + w_1\otimes w_2\otimes  \dots \otimes X w_r.
\end{aligned}
\ee
We denote the associated representations of both $\GL(V)$ and $\fg(V)$ on  $V^{\otimes_\Lambda r}$
by $\rho_r$. The first fundamental theorem for $\GL(V)$ is concerned with $\End_{\GL(V)}(V^{\ot_{\Lambda}r})$.

Define a $\Lambda$-linear action $\varpi_r$
of the symmetric group $\Sym_r$ of degree $r$ on $T^r(V)$ as follows. If $s_i=(i,i+1)$, $1\le i\le r-1$,
are the simple reflections which generate  $\Sym_r$, then for all $i$, we define $\varpi_r(s_i)$ by
\begin{eqnarray}\label{eq:sym}
\varpi_r(s_i): w \mapsto  w_1\otimes\dots \otimes \tau( w_i\otimes w_{i+1})\otimes \dots \otimes  w_r,
\end{eqnarray}
where $w=w_1\ot\dots\ot w_r\in T^r(V)$.
The group ring $\Lambda\Sym_r=\C\Sym_r\otimes_{\C}\Lambda$ is an associative superalgebra,
with $\C\Sym_r$ regarded as purely even. Extend $\Lambda$-linearly  the representation $\varpi_r$ of $\Sym_r$
to obtain an action of the superalgebra $\Lambda\Sym_r$.

It is evident that the actions of $\Lambda\Sym_r$ and  $\GL(V)$ on $T^r(V)$
commute with each other.
Thus we have a homomorphism of $\Z_2$-graded algebras
\be\label{eq:varpi}
\varpi_r:\Lambda\Sym_r\to\End_{\GL(V)}(V^{\ot_{\Lambda} r}),
\ee
where $\End_{\GL(V)}(V^{\otimes_\Lambda r})
= \{\phi\in\End_\Lambda(V^{\otimes_\Lambda r}) \mid g\phi = \phi g, \ \forall g\in \GL(V) \}$.

The following well-known result is the super analogue of Schur-Weyl duality. It is the FFT
for $\GL(V)$. A proof may be found in \cite[Theorem 3.2]{LZ6}, but the result goes back to
the physics literature (see e.g., \cite{DJ, BB}) and the papers \cite{S0, S1, BR}.
\begin{theorem}\label{thm:BR} (FFT for $\GL(V)$) The homomorphism \eqref{eq:varpi} is surjective.
That is, $\End_{\GL(V)}(V^{\otimes_\Lambda r}) = \varpi_r(\Lambda\Sym_r)$.
\end{theorem}

\subsection{Invariant theory for $\GL(V)$-the second fundamental theorem} The second fundamental theorem
for $\GL(V)$ describes the kernel of the surjective homomorphism $\varpi_r$ of Theorem \ref{thm:BR}.

Recall that $V=V_\C\ot_\C\Lambda$, where $\sdim(V)=(m|\ell)$. The following result is an easy consequence of
\cite[Theorem 3.20]{BR}.

\begin{theorem}\label{thm:BR2}
If $r\leq m\ell+m+\ell$, then the homomorphism $\varpi_r$ is an isomorphism. If $r>m\ell+m+\ell$,
then the kernel of $\varpi_r$ is the (two-sided) ideal of $\Lambda\Sym_r$ generated by the Young
symmetriser of the partition with $m+1$ rows and $\ell+1$ columns.
\end{theorem}

The kernel is therefore generated by an idempotent which is explicitly described as follows. Consider the
$(m+1)\times(\ell+1)$ array of integers below, which form a standard tableau.
\[
\begin{matrix}
1  &  2  &  \dots  &  \ell+1\\
\ell+2  & \ell+3  & \dots & 2\ell+2\\
\dots      &   \dots       &    \dots  &  \dots   \\
\dots      &   \dots       &    \dots  &  \dots   \\
m\ell+m+1     &   m\ell+m+2        &    \dots  &  m\ell+m+\ell+1    \\
\end{matrix}
\]

\noindent Let $R$ and $C$ be the subgroups of $\Sym_{m\ell+m+\ell+1}$ (regarded as the subgroup of $\Sym_r$
which permutes the first $m\ell+m+\ell+1$ numbers) which stabilise the rows and columns of the array respectively.
Thus
$$
\begin{aligned}
R=\Sym\{1,2,&\dots,\ell+1\}\times\Sym\{ \ell+2  , \ell+3  , \dots , 2\ell+2\}\times\dots\\
&\dots\times\Sym\{m\ell+m+1,   m\ell+m+2,    \dots,  m\ell+m+\ell+1\},\\
\end{aligned}
$$
while
$$
\begin{aligned}
C=\Sym\{1,\ell+2,&\dots,m\ell+m+1\}
\times\Sym\{ 2  , \ell+3  , \dots , m\ell+m+2\}\times\dots\\
&\dots\times\Sym\{\ell+1,   2\ell+2,    \dots,  m\ell+m+\ell+1\},\\
\end{aligned}
$$
where $\Sym\{X\}$ denotes the group of permutations of the set $X$.

Then in the group ring $\Lambda\Sym_{m\ell+m+\ell+1}\subseteq \Lambda\Sym_r$, let $e=e(m,\ell)$ be the (even)
element defined by
\be\label{eq:eml}
e(m,\ell)=\left(\sum_{\pi\in R}\pi\right)\left(\sum_{\sigma\in C}\ve(\sigma)\sigma\right)=\alpha^+(R)\alpha^-(C),
\ee
where $\ve$ is the sign character of $\Sym_r$, and for any subset $H\subseteq \Sym_r$, we write $\alpha^+(H)$ (resp. $\alpha^-(H)$)
for the element $\sum_{h\in H}h$ (resp. $\sum_{h\in H}\ve(h)h$) of $\C\Sym_r\subseteq \La\Sym_r$.

It is known that $(|R|!|C|!)\inv e(m,\ell)$ is a primitive idempotent in $\Lambda\Sym_{m\ell+m+\ell+1}$.
It is also well known that $\Lambda\Sym_r=\oplus_{\mu}I(\mu)$, where $\mu$ runs over the partitions of $r$,
and $I(\mu)$ is a simple ideal of $\Lambda\Sym_r$ for each $\mu$. In this notation, the ideal of $\Lambda\Sym_r$
which is generated by $e(m,\ell)$ is the sum of the $I(\mu)$ over those partitions $\mu$ which
contain an $(m+1)\times(\ell+1)$ rectangle.

\begin{corollary}\label{cor:sftgl}
If $r\leq m\ell+m+\ell$ then $\Ker(\varpi_r)=0$. Otherwise,
$\Ker(\varpi_r)=\oplus_\mu I(\mu)$ over those partitions $\mu$ of $r$ which contain
a rectangle of size $(m+1)\times(\ell+1)$.
\end{corollary}

\section{Invariant theory for the orthosymplectic supergroup}

In this section we take $\ell=2n$ and $G$ to be the orthosymplectic group $G=\OSp(V)\subseteq\GL(V)$
as in Definition \ref{def:sg}. Clearly $G$ acts on $T^r(V)$, and $\End_G(T^r(V))$ is a superalgebra
which contains $\End_{\GL(V)}(T^r(V))$. We aim to describe this algebra.

\subsection{The first fundamental theorem for $\OSp(V)$-first formulation} We need elements of $\End_G(T^r(V))$
which are not in $\End_{\GL(V)}(T^r(V))$. Suppose $r=2$ and consider the following element $c_0\in V\ot V$.
Let $(e_a),(e_a^*)$ ($a=1,2,\dots,m+2n$) be a pair of dual $\C$-bases of $V_\C$ in the sense that $(e_a^*,e_b)=\delta_{ab}$
(the Kronecker delta) for all $a,b$. The element $c_0=\sum_a e_a\ot e_a^*\in V\ot V$ is independent of the
basis chosen, and is $G$-invariant. Define $\gamma\in\End_G(V^{\ot 2})$ by $\gamma(v\ot w)=(v,w)c_0$.
This permits us to define elements $\gamma_i\in\End_G(V^{\ot r})$ ($i=1,\dots,r-1$) by
$\gamma_i=\underbrace{\id_V\ot\dots\ot\id_V}_{i-1}
\ot\gamma\ot\underbrace{\id_V\ot\dots\ot\id_V}_{r-i-1}$, with $\gamma$ acting on the $i,i+1$ factors.

The following result is equivalent to \cite[Cor. 5.7]{LZ6}.

\begin{theorem}\label{thm:fft-osp1}
The superalgebra $\End_G(T^r(V))$ is generated by the image of the homomorphism $\varpi_r$
of \eqref{eq:varpi}, together with $\gamma_1,\dots,\gamma_{r-1}$.
\end{theorem}

This statement may be reformulated in terms of the Brauer algebra (cf. \cite[\S 5.2]{LZ6}).
\begin{corollary}\label{cor:fft-ospbr}
Let $B_r(m-2n)$ be the Brauer algebra on $r$ strings, with parameter $m-2n$, with generators
$s_i,e_i$ ($i=1,\dots,r-1$ (see \cite[\S 5.2]{LZ6}). Then there is a surjective homomorphism
$B_r(m-2n)\lr\End_G(T^r(V))$, such that $s_i\mapsto\varpi_r(s_i)$ and $e_i\mapsto \gamma_i$
for all $i$.
\end{corollary}

\subsection{The first fundamental theorem-second formulation} The formulations given above of the first fundamental theorem
are statements about homomorphisms of non-commutative associative algebras. There is a second formulation
in terms of multilinear maps, which is equivalent to the first, but which will be more convenient for our purpose here.
Recall that $V^*=\Hom_\Lambda(V,\Lambda)$ is a free $\Lambda$-module of superdimension $(m|\ell)$. The group
$\GL(V)$ acts on $V^*$ via $g\phi(v):=\phi(g\inv v)$ (for $\phi\in V^*$, $g\in\GL(V)$ and $v\in V$).
For the discussion of the orthosymplectic case, we shall require that $\ell=2n$ is even.

\subsubsection{The case of $\GL(V)$} For the moment, we again take $\sdim(V_\C)=(m|\ell)$. To relate the two formulations, we
shall need the canonical isomorphism $\xi:V\ot V^*\lr\End_\La(V)$,
given by $\xi(v\ot\phi)(w)=v\phi(w)$. This map respects the $\GL(V)$ action, where $\GL(V)$ acts on $\End_\La(V)$
by conjugation. Consider the $\La$-module $\left(T^r(V^*)\ot_\La T^s(V)\right)^*$; this has an obvious $\GL(V)$
action, and if $r=s$, we have among the $\GL(V)$-invariant elements the functions $\delta_\pi$ ($\pi\in\Sym_r$),
defined in \eqref{eq:defdelta} below. Up to sign, we have
\be\label{eq:dpi}
\delta_\pi(\phi_1\ot\dots\ot\phi_r\ot v_1\ot\dots\ot v_r)=\pm\phi_1(v_{\pi 1})\phi_2(v_{\pi 2})\dots\phi_r(v_{\pi r}).
\ee

\begin{theorem}\label{thm:fft-gl2} With notation as above we have
\begin{eqnarray}
{\left(T^r(V^*)\ot_\La T^s(V)\right)^*}^{\GL(V)}=
\left\{
\begin{array}{c l}
0, &\text{ if }r\neq s\\
\sum\limits_{\pi\in\Sym_r}\La\delta_\pi, &\text{ if } r=s.
\end{array}
\right.
\end{eqnarray}
\end{theorem}

\subsubsection{Equivalence of the two formulations}
The equivalence between Theorems \ref{thm:fft-gl2} and \ref{thm:BR} is easily deduced from the following commutative
diagram, in which all maps respect the action of $\GL(V)$. The key is to identify the image of $\pi\in\Sym_r$ under $\alpha_r$
\begin{eqnarray}\label{eq:diag1}
\begin{aligned}
\xymatrix{
{\La\Sym_r}\ar[d]^{\varpi_r}\ar[r]^{\alpha_r\phantom{XXXXX}}&\left(T^r(V^*)\ot_\La T^r(V)\right)^* \\
\End_\La(T^r(V))&\ar[l]_{f^{(r)}}T^r(V)\ot T^r(V^*)\ar[u]^{e^{(r)}}_\wr.
}
\end{aligned}
\end{eqnarray}


To describe these maps explicitly, the following notation is useful.
If $\ul{a}, \ul{b}\in(\Z/2\Z)^r$, define
$J(\ul{a},\ul{b})\in\Z/2\Z$ by $J(\ul {a},\ul {b})=a_r(b_1+\dots+b_{r-1})+a_{r-1}(b_1+\dots+b_{r-2})+\dots+a_2b_1$.
We then also have $J(\ul{a},\ul{b})=b_1(a_2+\dots+a_r)+b_2(a_3+\dots+a_r)+\dots+b_{r-1}a_r=\sum_{i>j}a_ib_j$.

We shall apply this notation as follows. If $\ul v=v_1\ot\dots\ot v_r\in T^r(V)$ or
$\ul\phi=\phi_1\ot\dots\ot\phi_r\in T^r(V^*)$, we write $p(\ul v)$ for the parity sequence
$([v_1],[v_2],\dots,[v_r])\in(\Z/2\Z)^r$, and similarly for $p(\ul\phi)$.

The maps $e^{(r)}$ and $f^{(r)}$ are canonical $\GL(V)$-isomorphisms of graded $\La$-modules.
They are defined as follows. Let $\ul v=v_1\ot\dots\ot v_r\in T^r(V)$ and $\ul\phi=\phi_1\ot\dots\ot\phi_r\in T^r(V^*)$.
Then
$$
f^{(r)}(\ul v\ot\ul\phi): w_1\ot\dots\ot w_r\mapsto (-1)^{J(p(\ul\phi), p(\ul w))}\ul v \phi_1(w_1)\dots \phi_r(w_r).
$$

Similarly, for $\ul v,\ul w\in T^r(V)$ and $\ul\phi,\ul\psi\in T^r(V^*)$, we have
$$
e^{(r)}(\ul v\ot\ul \phi):\ul\psi\ot\ul w
\mapsto (-1)^{c(\ul v,\ul\phi,\ul \psi,\ul w)}
\psi_1(v_1)\dots\psi_r(v_r)\phi_1(w_1)\dots \phi_r(w_r),
$$
where $c(\ul v,\ul\phi,\ul \psi,\ul w)= [\ul\psi]([\ul v]+[\ul\phi])+J(p(\ul\psi),p(\ul v))+J(p(\ul\phi),p(\ul w))$.

Our goal is to identify $\alpha_r(\pi)$, where $\pi\in\Sym_n$. To do this, we shall identify $\varpi_r(\pi)$ and
use the isomorphisms above. For the former, we need the following definition.
\begin{definition}\label{def:n}
For $\ul v=v_1\ot\dots\ot v_r\in T^r(V)$ and $\sigma\in\Sym_r$, define
$n(\sigma,\ul v):=\sum_{(i,j)\in N(\sigma)}[v_i][v_j]\in\Z/2\Z$, where
$N(\sigma)=\{(i,j)\mid 1\leq i<j\leq r,\;\;\sigma(i)>\sigma(j)\}$.
\end{definition}

We then have
\begin{lemma}\label{lem:vpaction}
For $\pi\in\Sym_r$, the element $\varpi_r(\pi)\in\End_\La(T^r(V))$ takes
$\ul v=v_1\ot\dots\ot v_r$ to $(-1)^{n(\pi\inv,\ul v)}v_{\pi\inv 1}\ot\dots v_{\pi\inv r}$.
\end{lemma}

We next identify the element of $T^r(V)\ot T^r(V^*)$ which corresponds to $\varpi_r(\pi)\in\End_\La(T^r(V))$
under the isomorphism $f^{(r)}$. For this, we take a homogeneous basis $e_1,\dots,e_{m+\ell}$ of $V_\C$,
where $[e_i]=\bar0$ for $1\leq i\leq m$, and $[e_i]=\bar1$ otherwise. 
Let $\ve_1,\dots,\ve_{m+\ell}$ be the dual basis of $V_\C^*$. Then $(e_i)$ and $(\ve_i)$ are homogeneous bases of
the $\La$-modules $V$ and $V^*$ respectively, and are dual in the sense that $\ve_i(e_j)=\delta_{ij}$.  As usual, we write $[e_i]=[i]=[\ve_i]$
for the parity of these elements.

\begin{lemma}\label{lem:epi}
The element $E_\pi\in T^r(V)\ot T^r(V^*)$ which corresponds to $\varpi_r(\pi)\in\End_\La(T^r(V))$
is given by
$$
E_\pi=\sum_{i_1,\dots,i_a=1}^r(-1)^{J(\pi,(i_a))}e_{i_{\pi\inv 1}}\ot e_{i_{\pi\inv 2}}\ot\dots\ot
e_{i_{\pi\inv r}}\ot \ve_{i_{1}}\ot \ve_{i_{2}}\ot \dots\ot \ve_{i_{r}},
$$
where
$$
J(\pi,(i_a))=\sum_{\overset{1\leq a<b\leq r}{\pi\inv(a)<\pi\inv(b)}}[i_a][i_b].
$$
\end{lemma}
\begin{proof}
It is straightforward to check that $f^{(r)}(E_\pi)$ acts as $\varpi_r(\pi)$ on each element
of the form $e_{j_1}\ot\dots\ot e_{j_r}$. Since these elements form a basis of $T^r(V)$, the result
follows.
\end{proof}
\begin{definition}\label{def:brace}
Define the $\La$-bilinear non-degenerate pairing $\langle -,-\rangle:T^r(V^*)\times T^r(V)\lr\La$ by
$$
\langle \phi_1\ot\dots\ot \phi_r, v_1\ot\dots\ot v_r\rangle=(-1)^{J(p(\ul\phi),p(\ul v))}\phi_1(v_1)\dots\phi_r(v_r).
$$
It is a straightforward exercise to show that for $\pi\in\Sym_r$, we then have, for $\ul\phi\in T^r(V^*)$ and $\ul v\in T^r(V)$,
\be\label{eq:slide}
\langle \varpi_r(\pi)(\ul\phi),\ul v\rangle=\langle\ul\phi,\varpi_r(\pi\inv)(\ul v)\rangle.
\ee
We now define the functions $\delta_\pi\in \left((T^r(V^*)\ot T^r(V))^*\right)^{\GL(V)}$ by
\be\label{eq:defdelta}
\delta_\pi(\ul\psi\ot\ul w):=\langle\ul\psi,\varpi_r(\pi)(\ul w)\rangle,
\ee
for $\ul\psi\in T^r(V^*)$ and $\ul w\in T^r(V)$.
\end{definition}

\begin{lemma}\label{lem:epie}
For any elements $\ul\psi\in T^r(V^*)$ and $\ul w\in T^r(V)$, we have
$$
e^{(r)}(E_\pi)(\ul\psi\ot\ul w)=\delta_\pi(\ul\psi\ot\ul w).
$$
\end{lemma}
\begin{proof}
Since the statement is bilinear in $\ul\psi$ and $\ul w$ it suffices to prove it for $\ul\psi=\ve_{i_1}\ot\dots\ot\ve_{i_r}$
and $\ul w=e_{k_1}\ot\dots\ot e_{k_r}$. It is then clear that both sides are equal to zero unless $i_a=k_{\pi\inv a}$
for $a=1,2,\dots,r$. It remains only to check the sign, and for this we note that we have the following equation in
$\Z/2\Z$. For any sequence $([i_a])=[i_1],\dots, [i_r]\in(\Z/2\Z)^r$,
we have
$$
\begin{aligned}
J(\pi,(i_a))+n(\pi\inv,e_{i_1}\ot\dots\ot e_{i_r})&=J(([i_a]),([i_a]))\\
&=J(([i_{\pi\inv a}]),([i_{\pi\inv a}]))=\sum_{a\neq b}[i_a][i_b].\\
\end{aligned}
$$
\end{proof}
\begin{corollary}\label{cor:equ}
The functions
$\delta_\pi$ ($\pi\in\Sym_r$) span $\left((T^r(V^*)\ot T^r(V))^*\right)^{\GL(V)}$.
\end{corollary}

This is clear from the commutativity of the diagram \eqref{eq:diag1}.

\subsubsection{The case of $\OSp(V)$} The second formulation of the FFT for $G=\OSp(V)$ describes the
$\Lambda$-module $\left(T^r(V)^*\right)^G=T^r(V^*)^G$.
Here $V$ is the orthosymplectic $\La$-superspace corresponding to $V_\C$,
with $\sdim(V_\C)=(m|2n)$. As usual, the non-degenerate orthosymplectic form \eqref{eq:laform} on $V$ will
be denoted $(-,-)$.

If $r=2d$ is even and $\pi\in\Sym_{2d}$, then we have the element $\kappa_\pi\in\left(T^r(V)^*\right)^G$,
defined as follows.

First define $\langle-\rangle_0\in T^{2d}(V)^*$: for $\ul v=v_1\ot\dots\ot v_{2d}\in T^{2d}(V)$,
$$
\langle \ul v\rangle_0=(v_1,v_2)(v_3,v_4)\dots(v_{2d-1},v_{2d}).
$$
\begin{definition}\label{def:kappa} For $\pi\in\Sym_{2d}$ we define $\kappa_\pi\in T^{2d}(V)^*$ by
\be\label{eq:kappa}
\begin{aligned}
\kappa_\pi(v_1\ot&\dots\ot v_{2d}):=\langle\varpi_{2d}(\pi)(\ul v)\rangle_0\\
=&(-1)^{n(\pi\inv,\ul v)}(v_{\pi\inv(1)},v_{\pi\inv(2)})(v_{\pi\inv(3)},v_{\pi\inv(4)})\dots(v_{\pi\inv(2d-1)},v_{\pi\inv(2d)}).\\
\end{aligned}
\ee
\end{definition}

The second formulation of the FFT for $\OSp(V)$ is as follows.

\begin{theorem}\label{thm:fft-osp2}\cite[Theorem 4.3]{LZ6}
Maintaining the notation above, let $W=T^r(V^*)$ and $G=\OSp(V)$. If $r$ is odd, then $W^G=0$.
If $r=2d$ is even, then $W^G=\sum_{\pi\in\Sym_{2d}}\La\kappa_\pi$.
\end{theorem}

Of course we may have $\kappa_\pi=\kappa_{\pi'}$ for distinct permutations $\pi,\pi'\in\Sym_{2d}$, and this
permits us to formulate a slightly sharper version of the FFT for $\OSp(V)$ as follows.

\begin{definition}\label{def:c}
Let $C$ be the centraliser 
of the involution $(12)(34)\dots(2d-1,2d)$ in $\Sym_{2d}$. It is the semidirect product
$C=\Sym_d\ltimes (\zt)^d$.
\end{definition}

The quotient $\Sym_{2d}/C$
may be identified with the set of partitions $\{1,2,\dots,2d\}=\{i_1,j_1\}\amalg\{i_2,j_2\}\amalg\dots\amalg\{i_d,j_d\}$
of $\{1,2,\dots,2d\}$ into pairs. This is because $\Sym_{2d}$ evidently acts transitively on such partitions, and
$C$ is the stabiliser of the partition $\{1,2,\dots,2d\}=\{1,2\}\amalg\{3,4\}\amalg\dots\amalg\{2d-1,2d\}$.
This set of partitions is also evidently in
canonical bijection with the set of Brauer diagrams from $2d$ to $0$ (see \cite[Definitions 2.1 and 2.3]{LZ5}), and the action
of $\Sym_{2d}$ on $B_{2d}^0$ is precisely right multiplication by the diagrams, where a permutation $\pi\in\Sym_{2d}$
is thought of as a Brauer diagram $\pi:2d\to 2d$ in $B_{2d}^{2d}$.

We shall freely make use of the language of the Brauer category and refer the reader to \cite{LZ5} for details.
In particular, for non-negative integers $p,q$ $B_p^q$ denotes the $\La$ module of Brauer diagrams $D:p\to q$.
Specifically, the partitions $D$ above may be thought of as morphisms from $2d$ to $0$ in the Brauer category $\CB(m-2n)$
(see Fig. 1).

\setlength{\unitlength}{0.30mm}
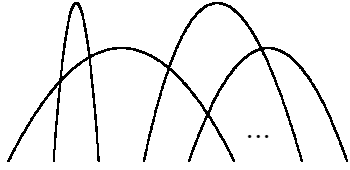
\begin{figure}[h]
\begin{center}
\begin{picture}(200, 80)(0,50)
\qbezier(30, 60)(80, 160)(130, 60)
\qbezier(50, 60)(60, 200)(70, 60)

\put(135, 70){$...$}

\qbezier(90, 60)(120, 200)(160, 60)

\qbezier(110, 60)(145, 160)(180, 60)
\end{picture}
\end{center}
\caption{A diagram $D: 2d\longrightarrow 0$}
\label{Fig5}
\end{figure}

\setlength{\unitlength}{0.35mm}
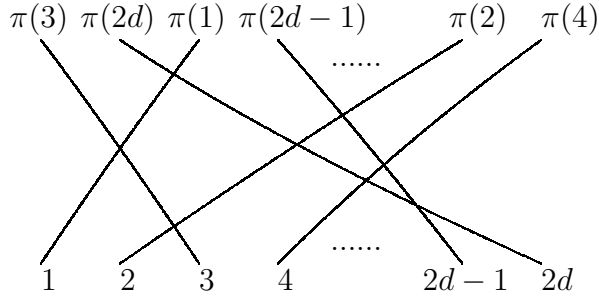
\begin{figure}[h]
\begin{center}
\begin{picture}(250, 120)(0,0)

\put(20, 0){$1$}
\put(50, 0){$2$}
\put(80, 0){$3$}
\put(110, 0){$4$}

\put(130, 15){......}

\put(165, 0){$2d-1$}
\put(210, 0){$2d$}

\put(130, 85){......}

\qbezier(20, 10)(50, 55)(80, 95)
\put(68, 100){$\pi(1)$}

\qbezier(50, 10)(115, 55)(180, 95)
\put(175, 100){$\pi(2)$}

\qbezier(80, 10)(50, 55)(20, 95)
\put(8, 100){$\pi(3)$}

\qbezier(110, 10)(155, 55)(210, 95)
\put(210, 100){$\pi(4)$}

\qbezier(180, 10)(145, 55)(110, 95)
\put(95, 100){$\pi(2d-1)$}

\qbezier(210, 10)(110, 55)(50, 95)
\put(35, 100){$\pi(2d)$}

\end{picture}
\end{center}
\caption{A permutation $\pi: 2d\to 2d$}
\label{Fig2}
\end{figure}

\begin{lemma}\label{lem:ctriv}
For any elements $\ul v\in T^{2d}(V)$ and $\sigma\in C$, we have $\langle \varpi_{2d}(\sigma)\ul v\rangle_0=\langle\ul v\rangle_0$.
\end{lemma}
\begin{proof}
The group $C$ is generated by the permutations $(2i-1,2i+1)(2i,2i+2)$ ($i=1,2,\dots,d-1$) together with  $(1,2)$.
It therefore suffices to prove the statement for these permutations and all $\ul v$, and this is a straightforward
exercise.
\end{proof}

\setlength{\unitlength}{0.35mm}
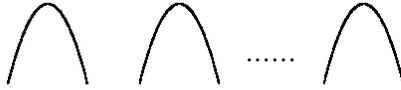
\begin{figure}[h]
\begin{center}
\begin{picture}(200, 30)(0,0)
\qbezier(30, 0)(45, 60)(60, 0)
\qbezier(80, 0)(95, 60)(110, 0)
\put(120, 8){$......$}
\qbezier(150, 0)(165, 60)(180, 0)

\end{picture}
\end{center}
\caption{The diagram $D_0: 2d\to 0$}
\label{Fig3}
\end{figure}

\begin{corollary}\label{cor:kd}
Let $D_0\in B_{2d}^0$ be the diagram corresponding to the partitioning
$\{1,2,\dots,2d\}=\{1,2\}\amalg\{3,4\}\amalg\dots\amalg\{2d-1,2d\}$ and let $D\in B_{2d}^0$
be  any diagram. If $\pi\in\Sym_{2d}$ is such that $D=D_0\pi$, then $\kappa_\pi$ is independent of $\pi$,
i.e. depends only on $D$. 
\end{corollary}
\begin{proof}
For $\pi_1,\pi_2\in\Sym_{2d}$, we have $D_0\pi_1=D_0\pi_2$ if and only if $D_0\pi_1\pi_2\inv=D_0$,
i.e. if and only if $\pi_1=\sigma\pi_2$ for some $\sigma\in C$. But for any $\ul v\in T^{2d}(V)$,
we then have
$$
\kappa_{\pi_1}(\ul v)=\langle \varpi_{2d}(\sigma\pi_2)(\ul v)\rangle_0
=\langle\varpi_{2d}(\sigma)(\varpi_{2d}(\pi_2)\ul v)\rangle_0
=\langle\varpi_{2d}(\pi_2)\ul v\rangle_0
=\kappa_{\pi_2}(\ul v),
$$
the last equality being a consequence of the Lemma. This proves the corollary.
\end{proof}

\begin{definition}\label{def:kd}
For any Brauer diagram $D\in B_{2d}^0$, define $\kappa_D\in (T^{2d}(V^*))^G$ by $\kappa_D=\kappa_\pi$ for
any $\pi\in\Sym_{2d}$ such that $D=D_0\pi$. By Corollary \ref{cor:kd}, this is well defined.
\end{definition}

\begin{corollary}\label{cor:fft-osp2}
In the notation above, we have $T^{2d}(V^*)^{\OSp(V)}=W^G=\sum\limits_{D\in\cD}\La\kappa_D$.
\end{corollary}
\begin{proof}
For each permutation $\pi\in\Sym_{2d}$, we have $\kappa_\pi=\kappa_D$ where $D$ is the diagram $D_0\pi$.
The result is now immediate from Theorem \ref{thm:fft-osp2}.
\end{proof}

\subsection{The second fundamental theorem for $\GL(V)$-second formulation} Recall that when discussing
$\GL(V)$, we take $\sdim(V_\C)=(m|\ell)$.
We now interpret the second fundamental theorem for $\GL(V)$ in terms of the second formulation of the
first fundamental theorem. In the notation of Theorem \ref{thm:fft-gl2}, to describe the space $\left((T^r(V^*)\ot T^r(V))^*\right)^{\GL(V)}$,
we need to
describe all $\La$-linear relations among the functions $\delta_\pi$ (see \eqref{eq:defdelta}). The result is as follows.

\begin{theorem}\label{thm:sft-gl2} In the notation above,
all linear relations among the $\delta_\pi$ are consequences of those of the form
$$
\sum_{\pi\in\Sym_{r}}a_\pi\delta_\pi=0,
$$
where $\sum_{\pi\in\Sym_{r}}a_\pi\pi\in I(\mu)$ for some partition $\mu$ of $r$, which contains an
$(m+1)\times(\ell+1)$ rectangle.
\end{theorem}
\begin{proof}
In the diagram \eqref{eq:diag1}, $\alpha_r:\La\Sym_r\lr{\left(T^r(V^*)\ot_\La T^r(V)\right)^*}^{\GL(V)}$
is defined so as to make the diagram commute. By commutativity, $\Ker(\alpha_r)=\Ker(\varpi_r)=\sum I(\mu)$,
over those partitions $\mu$ which contain an $(m+1)\times(\ell+1)$ rectangle.

The commutativity of the diagram \eqref{eq:diag1}, together with Lemmas \ref{lem:vpaction} and \ref{lem:epie},
imply that we have an isomorphism of $\La$-modules $\Phi:\End_{\GL(V)}(T^r(V))\lr (T^r(V^*)\ot T^r(V))^{\GL(V)}$,
such that $\Phi(\varpi_r(\pi))=\delta_\pi$ for each permutation $\pi\in\Sym_r$. It follows that for elements $a_\pi\in\La$,
$\sum_{\pi\in\Sym_{r}}a_\pi\delta_\pi=0$ if and only if $\sum_{\pi\in\Sym_{r}}a_\pi\varpi_r(\pi)=0$, and by the
first formulation, the last relation holds if and only if
$\sum_{\pi\in\Sym_{r}}a_\pi\pi\in I(\mu)$ for some partition $\mu$ of $r$, which contains an
$(m+1)\times(\ell+1)$ rectangle.
\end{proof}

\section{The second fundamental theorem for $\OSp(V)$}

In this section we assume that $\sdim (V_\C)=(m|2n)$, and we maintain the standard notation $V=V_\C\ot_\C\La$, and
$G=\OSp(V)$. The orthosymplectic form on $V$ is denoted $(-,-)$, and the $\La$-linear involution $A\mapsto A^\dag$
on $\End(V)$ is defined by $(Av,w)=(-1)^{[A][v]}(v,A^\dag w)$ for all $v,w\in V$ and $A\in\End(V)$.

\subsection{Canonical identifications}
We recall the following facts from \cite[\S 3.4]{LZ6}. For $v\in V$, the element $\phi_v\in V^*$ is defined by
$\phi_v(w)=(v,w)$ ($v,w\in V$). Then for $g\in\GL(V)$, we have $g\phi_v=\phi_{\hat g v}$, where $\hat g=(g^\dag)\inv$.
Further, we have a canonical isomorphism of $\GL(V)$-modules $\zeta^*:V^*\ot V^*\lr \End(V)$, where
$\zeta^*(\phi_v\ot\phi_w):x\mapsto v(w,x)$ (for $x,v,w\in V$). Here the action of $\GL(V)$ on $\End(V)$ is given by
$g:A\mapsto \hat g Ag\inv$ for $g\in\GL(V)$ and $A\in\End(V)$.

Recall \cite[Lemma 3.22]{LZ6} that the isomorphism $\zeta^*$ satisfies $\zeta^*\circ \tau=^\dag\circ\zeta^*$,
where $\tau$ is the usual graded involution defined by $\tau(\phi\ot\psi)=(-1)^{[\phi][\psi]}\psi\ot\phi$.
That is, for $\phi,\psi\in V^*$, we have $\zeta^*(\tau(\phi\ot\psi))=(\zeta^*(\phi\ot\psi))^\dag$. It follows that
if $S^2(V^*)$ and $\wedge^2(V^*)$ are respectively the $+1$ and $-1$ eigenspaces of $\tau$,
and $\End_\La(V)^\pm$ are the $\pm 1$ eigenspaces of $^\dag$, then $\zeta^*$ defines isomorphisms
$S^2(V^*)\to\End_\La(V)^+$ and $\wedge^2(V^*)\to\End_\La(V)^-$.

We shall make crucial use of the following map.

\begin{definition}\label{def:omega}
Define the map $\omega:\End_\La(V)\to\End_\La(V)^+$ by $\omega(A)=A^\dag A$. We shall generally restrict $\omega$
to $\End_\La(V)_{\bar0}$ without changing the notation.
\end{definition}

The following observation will be useful. Let $r=2d$ be even. Then we have the canonical $\GL(V)$-isomorphism
\be\label{eq:ci}
T^d(\zeta^*)\ot T^{2d}(\id_V):T^{2d}(V^*)\ot T^{2d}(V)\lr T^d(\End(V))\ot T^{2d}(V).
\ee

\begin{lemma}\label{lem:endinv}
We have the following commutative diagram
\begin{displaymath}
\xymatrix{
        &T^{2d}(V^*)\ot T^{2d}(V)\ar[dr]_{\langle-,-\rangle} \ar@{->}[rr] &&T^d(\End(V))\ot T^{2d}(V)\ar[dl]^{\mu}\\
							&&\La&
}
\end{displaymath}
where the horizontal arrow is as in \eqref{eq:ci},
$\langle-,-\rangle\in \left((T^{2d}(V^*)\ot T^{2d}(V))^*\right)^{\GL(V)}$ is the form defined in
Definition \ref{def:brace} above,  and $\mu(A_1\ot\dots\ot A_d\ot v_1\ot\dots\ot v_{2d})=
(-1)^{J(\ul A,\ul v)}(v_1,A_1v_2)\dots (v_{2d-1},A_dv_{2d})$ for $A_i\in\End(V)$ and $v_j\in V$,
with $J(\ul A,\ul v)=\sum_{i=1}^d[A_i](\sum_{j=1}^{2i-1}[v_j])$.
\end{lemma}

\begin{proof}
Note first that it is easily verified that $\mu\in  \left((T^{2d}(V^*)\ot T^{2d}(V))^*\right)^{\GL(V)}$.
It is straightforward to verify the commutativity of the diagram on elements of the form
$\phi_{v_1,w_1}\ot\dots\phi_{v_d,w_d}\ot  x_1\ot\dots\ot x_{2d}$, where $\phi_{v,w}=\zeta^*(\phi_v\ot\phi_w)$.
The result now follows by $\La$-linearity.
\end{proof}

\begin{corollary}\label{cor:dpi2}
For $\pi\in\Sym_{2d}$, the element $\delta_\pi$ is realised on $\End(V))\ot T^{2d}(V)$ by the formula
$$
\begin{aligned}
\delta_\pi(A_1\ot\dots&\ot A_d\ot w_1\ot\dots\ot w_{2d})\\&=(-1)^{J(\ul A,\ul w_{\pi\inv})+n(\pi\inv,\ul w)}
(w_{\pi\inv 1},A_1w_{\pi\inv 2})\dots(w_{\pi\inv(2d-1)},A_dw_{\pi\inv(2d)}),\\
\end{aligned}
$$
where $n(\pi\inv,\ul w)$ is as in Definition \ref{def:n} and $\ul w_{\pi\inv}=w_{\pi\inv 1}\ot\dots\ot w_{\pi\inv(2d)}$.
\end{corollary}
\begin{proof}
From \eqref{eq:defdelta}, we have $\delta_\pi(\ul\phi\ot\ul w)=\langle\ul\phi,\varpi_r(\pi)\ul w\rangle$
for $\ul\phi\in T^{2d}(V^*)$. The statement is now immediate from Lemmas \ref{lem:vpaction} and \ref{lem:endinv}.
\end{proof}

\subsection{Some super ($\zt$-graded) geometry} We shall recall and complement some notions concerning graded-commutative algebraic
geometry. Background for the material in this section may be found in \cite[\S 3.2]{LZ6}.

Let $\Lambda$ be the infinite dimensional Grassmann algebra
as above, and write its generators as $e(1),e(2),\dots$. For any positive integer $N$, $\La(N)$ denotes the subalgebra of $\La$
generated by $e(1),e(2),\dots,e(N)$. Recall that $\La$ has a natural grading $\La=\oplus_{k=0}^\infty \La^k$, where $\La^k$ is the
subspace with basis $\{e(i_1)e(i_2)\dots e(i_k)\mid i_1<i_2<\dots<i_k\}$ and is a $\zt$-graded algebra, with
$\La_{\bar0}=\oplus_{k\text{ even}}\La^k$ and $\La_{\bar1}=\oplus_{k\text{ odd}}\La^k$.

Let $M_\C$ be a $\zt$-graded $\C$-vector space with $\sdim(M_\C)=(k|l)$. We then have the superspace $M:=M_\C\ot_\C\La$. The (graded) symmetric algebra
$S(M)$ is defined as $S(M)=T(M)/I(M)=\oplus_{r=0}^\infty S^r(M)$, where $T(M)=\oplus_{j=0}^\infty M^{\ot_\La r}$ is the tensor algebra,
and $I(M)$ is the ideal of $T(M)$ generated by elements of the form $v\ot w- (-1)^{[v][w]}w\ot v\in T^2(V)$. It was shown in \cite[loc. cit.]{LZ6}
that $S^r(M)$ may be identified with $T^r(M)/\Sym_r$, where $\Sym_r$ acts via $\varpi_r$ as above.

If, for any set $U$, $\CF(U,\La)$ denotes the $\La$-module of functions $f:U\to\La$, then we have a map $F^r:S^r(M^*)\lr\CF(M_{\bar0},\La)$
which was shown in \cite[loc. cit.]{LZ6} to be injective. The map is defined as follows. 
For $\phi_1,\phi_2,\dots,\phi_r\in M^*$, denote
by $\ol{\phi_1\ot\dots\ot\phi_r}$ the image in $S^r(M^*)$ of $\phi_1\ot\dots\ot\phi_r\in T^r(M^*)$. The for any $m\in M_{\bar0}$, we have 
$F^r(\ol{\phi_1\ot\dots\ot\phi_r})(m):=\langle\phi_1\ot\dots\ot\phi_r,m\ot m\ot\dots\ot m\rangle$, where the bracket is defined in Definition
\ref{def:brace}. This definition depends only on $\ol{\phi_1\ot\dots\ot\phi_r}\in S^r(M^*)$, not ${\phi_1\ot\dots\ot\phi_r}$.
The image $F^r(S^r(M^*)):=\CP^r[M_{\bar0}]$, the $\La$-module of polynomial functions of degree $r$ on $M_{\bar0}$. Let $\CP[M_{\bar0}]:=\oplus_{r= 0}^\infty\CP^r[M_{\bar0}]$, 
which forms an algebra, the algebra of polynomial functions on $M_{\bar0}$.

Now let $b_1,\dots,b_{k+l}$ be a homogeneous basis of $M_\C$, such that $b_1,\dots,b_k\in (M_\C)_{\bar0}$
while $b_{k+1},\dots,b_{k+l}\in (M_\C)_{\bar1}$. Let $X_1,\dots,X_{k+l}$ be the dual basis of $M_\C$. Then
$(b_i),(X_i)$ are also dual homogeneous bases of the $\La$-module $M$. It was shown in \cite[\S 3.2]{LZ6} that
a polynomial $f\in\CP^r[M_{\bar0}]$ may be represented (uniquely) in the form
\be\label{eq:poly}
f=\sum_{m_1+\dots+m_{k+l}=r}\la_{m_1m_2\dots m_{k+l}}X_1^{m_1}X_2^{m_2}\dots X_{k+l}^{m_{k+l}},
\ee
where the sum is over non-negative itegers $m_i$,  $\la_{m_1m_2\dots m_{k+l}}\in\La$, and $m_j=0\text{ or }1$ if $k+1\leq j\leq k+l$.
The value of this function $f$ at $m=\sum_{j=1}^{k+l}b_j\mu_j\in M_{\bar0}$ is given by
\be\label{eq:fm}
f(m)=\sum_{m_1+\dots+m_{k+l}=r}\la_{m_1m_2\dots m_{k+l}}\mu_1^{m_1}\mu_2^{m_2}\dots \mu_{k+l}^{m_{k+l}}.
\ee

A subset $U\subseteq M_{\bar0}$ is said to be {\em dense} if, for any polynomial function
$f\in\CP[M_{\bar0}]$, we have $f(U)=0$
implies that $f=0$.

We shall give a suffcient condition for a subset to be dense. Before stating it, we make some observations
about $M_{\bar0}$. Notice that $M_{\bar0}=(M_\C)_{\bar0}\ot_\C\La_{\bar0}\oplus (M_\C)_{\bar1}\ot_\C\La_{\bar1}$,
and that $(M_\C)_{\bar0}\ot_\C\La_{\bar0}\supseteq (M_\C)_{\bar0}=\C^k$ 
since $\La^0=\C$. Let $\Lambda(\text{$>$$N$})=\langle e(n)\mid n>N\rangle$ 
be the subalgebra of $\Lambda$ generated by $\{e(n)\mid n>N\}$. 
Now let us say that the subset $U\subseteq M_{\bar0}$ has property D if
\be\label{eq:d1}
u\in U \text { and }\beta\in\C\implies\beta u\in U,
\ee
and for each integer $N\geq 0$ there are elements
$\nu_1,\dots,\nu_l\in\Lambda(\text{$>$$N$})\cap\La_{\bar1}$ such that the products $\nu_{i_1}\nu_{i_2}\dots\nu_{i_j}$ with $i_1<\dots<i_j$ for all $j\le \ell$ are linearly independent (in particular, $\nu_1\nu_2\dots\nu_l\neq 0$)
and a Zariski dense subset $U^0\subseteq \C^k(=(M_\C)_{\bar0}\ot\La^0)$ such that for each element $(\beta_1,\dots,\beta_k)\in U^0$, we have
\be\label{eq:d2}
b_1\beta_1+\dots+b_k\beta_k+b_{k+1}\nu_1+b_{k+2}\nu_2+\dots+b_{k+l}\nu_l\in U.
\ee

\begin{lemma}\label{lem:dense}
Suppose $U\subseteq M_{\bar0}$ satisfies the conditions \eqref{eq:d1} and \eqref{eq:d2} above. Then $U$ is dense in $M_{\bar0}$.
\end{lemma}
\begin{proof}
Suppose $f\in\CP[M_{\bar0}]$ vanishes on $U$, and that $f=f_0+f_1+\dots+f_r$, where $f_i\in\CP^i[M_{\bar0}]$.
We need to show that $f$ vanishes on $M_{\bar0}$. Notice first that for $u\in U$ and $\beta\in\C$,
$f(\beta u)=f_0(u)+\beta f_1(u)+\beta^2f_2(u)+\dots+\beta^rf_r(u)$. Hence by the property \eqref{eq:d1}, $f(u)=0$ if and only if
$f_i(u)=0$ for all $i$. Hence we may assume that $f=f_r$ is homogeneous.

Now suppose that $f$ is as in \eqref{eq:poly}. Then there is an integer $N$ such that $\la_{m_1m_2\dots m_{k+l}}$ $\in\La(N)$.
Take $\nu=\sum_{j=1}^lb_{k+j}\nu_j$ as specified by the property \eqref{eq:d2}. Recall that $m_i=0$ or $1$ if $k+1\leq i\leq k+l$.
Fix a sequence $M_{k+1},\dots,M_{k+l}$ such that $\sum_j M_{k+j}=r-s\leq r$. We shall show that
$\la_{m_1\dots m_kM_{k+1}\dots M_{k+l}}=0$ for all sequences $m_1,\dots,m_k$ with $\sum_{i=1}^sm_i=s$.

But by choice of the $\nu_i$, it is clear that for all $(\beta_1,\dots,\beta_k)\in U^0$,
$$
\sum_{m_1+\dots+m_{k}=s}\la_{m_1\dots m_{k}M_{k+1}\dots M_{k+l}}\beta_1^{m_1}\beta_2^{m_2}\dots \beta_{k}^{m_{k}} =0.
$$
It follows since $U^0$ is Zariski dense in $\C^k$ that all coefficients $\la_{m_1\dots m_kM_{k+1}\dots M_{k+l}}$ are zero.
\end{proof}

We shall need the next two results later.

\begin{proposition}\label{prop:omegadense}
The image of the map $\omega:\End_\La(V)_{\bar0}\to \End_\La(V)_{\bar0}^+$ (see Definition \ref{def:omega}) above)
is dense in $\End_\La(V)_{\bar0}^+$.
\end{proposition}
\begin{proof}
The argument in the proof of \cite[Lemma 6.7]{LZ6} shows that
the image of $\omega$ satisfies the conditions \eqref{eq:d1} and \eqref{eq:d2}.
Hence the result is immediate from Lemma \ref{lem:dense}.
\end{proof}

\begin{proposition}\label{prop:symspan}
Let $M_\C$ be a $\zt$-graded complex vector space with $\sdim(M_\C)=(k|\ell)$, and let $M=M_\C\ot_\C\La$ be the corresponding superspace
over the Grassmann algebra. Let $B\subseteq M_{\bar0}$ be a subset satisfying the conditions

(i) $\La B=M$; that is, $B$ generates $M$ as $\La$-module.  

(ii) For any two distinct elements $b,c\in B$ there are infinitely many complex numbers $\la$ such that $b+\la c\in B$.

Then for each integer $r\geq 0$, the space $\left(S^r(M) \right)_{\bar0}$ is the $\La_{\bar0}$-span 
of the elements $\{b\ot b\ot\dots\ot b\mid b\in B\}$.
\end{proposition}
\begin{proof} 
This will be by induction on $r$, the case $r=1$ being trivial, since by the condition (i),
$\La_{\bar0}B=M_{\bar0}$, so that $M_{\bar0}$ is the $\La_{\bar0}$-span of $B$ (and 
$M_{\bar1}$ is the $\La_{\bar1}$-span of $B$).

Now it follows from \cite[Lemma 3.7]{LZ6} that $S^r(M)=\alpha^+(r)T^r(M)$, where $\alpha^+(r)=\sum_{\sigma\in\Sym_{r}}\sigma$, with 
$\Sym_r$ acting on $T^r(M)$ in the usual way through $\varpi_r$ (see \eqref{eq:sym}). Denote by $\Sigma(r)$ the $\La_{\bar0}$-span 
of $\{b\ot b\ot\dots\ot b\mid b\in B\}$. Then by the last remark, it suffices to show that for any element $m_1\ot m_2\ot\dots\ot m_r\in T^r(M)_{\bar0}$,
we have 
\be\label{proj}
\alpha^+(r)(m_1\ot m_2\ot\dots\ot m_r)\in \Sigma(r).
\ee

We shall prove \eqref{proj} by induction on $r$. Assume \eqref{proj} for a smaller number of factors; we consider separately 
two cases.

\noindent{\em Case 1: at least one of the factors $m_i$ lies in $M_{\bar0}$.} In this case, since for any element $\sigma\in\Sym_r$
we have $\alpha^+(r)\sigma=\alpha^+(r)$, we may clearly assume that $m_r\in M_{\bar0}$. Then 
$m_1\ot m_2\ot\dots\ot m_{r-1}\in T^{r-1}(M)_{\bar0}$, and by induction, we have 
$\alpha^+(r-1)(m_1\ot m_2\ot\dots\ot m_{r-1})\in\Sigma(r-1)$.

But from the coset decomposition of $\Sym_r$ with respect to $\Sym_{r-1}$, we have 
$$
\alpha^+(r)=\left(1+\sum_{i=1}^{r-1}(i,r)  \right)\alpha^+(r-1),
$$ 
where $(i,j)$ denotes the transposition of $i$ and $j$ in $\Sym_r$ and $\Sym_{r-1}$ is the subgroup of $\Sym_r$ which
permutes $\{1,\dots,r-1\}$.
The last two observations imply that $\alpha^+(r)(m_1\ot m_2\ot\dots\ot m_r)$ is a $\La_{\bar0}$-linear combination of elements of the 
form 
\be\label{form}
\left(1+\sum_{i=1}^{r-1}(i,r)  \right)(b\ot b\ot\dots\ot b\ot c),
\ee
where $b,c\in B$.

Now denote by $s_{i,j}(b,c)$ the sum of all tensors of the form $\dots\ot b\ot\dots\ot c\ot \dots\in T^r(M)$, where all factors are either $b$ or $c$,
and there are $i$ factors equal to $b$, and $j$ factors equal to $c$. Then the element in \eqref{form} is $s_{1,r-1}(b,c)$, and it will suffice
to show that if $b,c\in B$, then $s_{i,j}(b,c)\in\Sigma(r)$ for all $i,j$ such that $i+j=r$. To see this last point, suppose $\la\in\C$ is such that
$b+\la c\in B$. Then $\Sigma(r)\ni(b+\la c)\ot(b+\la c)\ot\dots\ot(b+\la c)=\sum_{j=0}^r\la^js_{i,j}(b,c)$. By taking $r+1$ distinct values $\la$
for which this relation holds, we obtain, by the invertibility of the van der Monde matrix, an equation for each element $s_{i,j}(b,c)$
as a $\C$-linear combination of the elements $(b+\la c)\ot(b+\la c)\ot\dots\ot(b+\la c)\in\Sigma(r)$. Thus $s_{r-1,1}(b,c)\in\Sigma(r)$,
and the proof in Case 1 is complete.

\noindent{\em Case 2: each factor $m_i\in M_{\bar1}$.} In this case, we must have $r$ even. Since $\alpha^+(r)$ is linear in each 
variable $m_i$ and since $M_{\bar1}=\La_{\bar1}B$, we may assume that for $i=1,2,\dots,r$, we have $m_i=\la_ib_i$, 
where $\la_i\in\La_{\bar1}$ and $b_i\in B$. Then
$$
\begin{aligned}
\alpha^+(r)(\la_1b_1\ot\dots\la_r b_r)=&\sum_{\sigma\in\Sym_r}\sigma(\la_1b_1\ot\dots\la_r b_r)\\
=&\sum_{\sigma\in\Sym_r}\ve(\sigma)\la_{\sigma (1)}\la_{\sigma(2)}\dots\la_{\sigma(r)} b_{\sigma(1)}\ot\dots \ot b_{\sigma(r)}\\
=&\sum_{\sigma\in\Sym_r}\ve(\sigma)^2\la_{1}\la_{2}\dots\la_{r} b_{\sigma(1)}\ot\dots\ot b_{\sigma(r)}\\
=&\la_{1}\la_{2}\dots\la_{r} \alpha^+(r) (b_1\ot\dots\ot b_r),\\
\end{aligned}
$$
where $\ve$ is the alternating character of $\Sym_r$. But since $r$ is even, $\la_1\dots\la_r\in\La_{\bar0}$;
finally,  observe that by Case 1, $ \alpha^+(r) (b_1\ot\dots\ot b_r) \in\Sigma(r)$, and the proof is complete.
\end{proof}

\begin{corollary}\label{cor:symspan}
Let $R\subseteq (\End_\La(V))_{\bar0}^+$ be the set of elements $A\in (\End_\La(V))_{\bar0}^+$ such that
$A=\omega(X)=X^\dag X$ for some $X\in (\End_\La(V))_{\bar0}$. Then $S^d((\End_\La(V))^+)_{\bar0}$
is generated as $\La_{\bar0}$-module by $\{A\ot A\ot\dots\ot A\mid A\in R\}$.
\end{corollary}
\begin{proof}
By Proposition \ref{prop:symspan} it suffices to prove that the set $R$ satisfies the conditions (i) and (ii) 
on the set $B$ in the statement. A close examination of the proof of \cite[Lemma 6.7]{LZ6} verifies these conditions.
\end{proof}

\subsection{Orthosymplectic invariants} Let $L\in \left(T^{2d}(V)^*\right)^G$, where $G=\OSp(V)$.
In \cite[\S 4.2]{LZ6} we showed that there is a unique element $F_L\in \CP^d[S^2(V^*)_{\bar0}]\ot T^{2d}(V)^*$
such that
\be\label{eq:fl}
F_L(\omega(A)\ot v_1\ot\dots\ot v_{2d})=L(Av_1\ot\dots\ot Av_{2d})
\ee
for any $A\in\End(V)_{\bar0}$, where $\omega(A)=A^\dag A\in \End_\La(V)_{\bar0}^+
\cong S^2(V^*)_{\bar0}$.
Note that generically (i.e. for invertible $X\in S^2(V^*)$), $X=\omega(A)=\omega(B)$ implies that
$B=gA$ for some $g\in\OSp(V)(=\{g\in\GL(V)\mid g^\dag=g\inv\})$. Hence by $\OSp(V)$-invariance,
$L(Av_1\ot\dots\ot Av_{2d})=L(Bv_1\ot\dots\ot Bv_{2d})$. Further, for $g\in\GL(V)$,
$\omega(Ag\inv)=\hat g\omega(A)g\inv$. It follows that $F_L\in \left(\CP^d[S^2(V^*)_{\bar0}]\ot T^{2d}(V)^*\right)^{\GL(V)}$.

\begin{remark}\label{rem:fl}
Note that since we have canonical $\GL(V)$-equivariant isomorphisms
$$
\CP^d[S^2(V^*)_{\bar0}]\cong S^d(S^2(V^*)^*)\cong  \left(S^d(S^2(V^*))\right)^*,
$$
$F_L$ may be thought of as an element of $\left(S^d(S^2(V^*))\ot T^{2d}(V)\right)^*$, and hence as a $\GL(V)$-invariant linear map
\be\label{eq:fllin}
F_L:S^d(S^2(V^*))\ot T^{2d}(V)\lr \La.
\ee
\end{remark}


Now there are canonical $\GL(V)$-equivariant surjections
\be\label{eq:sur}
T^d(\End_\La(V))\ot T^{2d}(V)\overset{\sim}{\lr}T^{2d}(V^*)\ot T^{2d}(V)\overset{{/C\ot\id}}{\lr}S^d(S^2(V^*))\ot T^{2d}(V),
\ee
where $C$ is the centraliser in $\Sym_{2d}$ of the element $(12)(34)\dots(2d-1,2d)$, acting on $T^{2d}(V^*)$ as usual
(i.e. through $\varpi_{2d}$). It follows that $F_L$ may be pulled back to $T^d(\End_\La(V))\ot T^{2d}(V)$ to give an element
$H_L\in \left( (T^d(\End_\La(V))\ot T^{2d}(V))^*    \right)^{\GL(V)}$.

The result of these observations is the following statement. In the statement below, we use the notation of \cite[\S 3.4]{LZ6},
so that $\End_\La(V)^{\pm}=\{A\in\End_\La(V)\mid A^\dag=\pm A\}$. Recall that under the isomorphism
$\zeta^*:V^*\ot V^*\simeq\End_\La(V)$,  $\End_\La(V)^{+}$ and $ \End_\La(V)^{-}$ correspond respectively to
$S^2(V^*)$ and $\wedge^2(V^*)$.

\begin{proposition}\label{prop:ltohl}
Given $L\in \left(T^{2d}(V)^*\right)^{G}$, where $G=\OSp(V)$, there is a uniquely defined element
$H_L\in \left( (T^d(\End_\La(V))\ot T^{2d}(V))^*    \right)^{\GL(V)}$ with the following properties.
\begin{enumerate}
\item The map $h:L\mapsto H_L$ is $\La$-linear  and injective
$$
h:(T^{2d}(V)^*)^G\lr \left( (T^d(\End_\La(V))\ot T^{2d}(V))^*\right)^{\GL(V)}.
$$
\item We have $H_L(\id_V\ot\dots\ot \id_V\ot\ul v)=L(\ul v)$ for $\ul v\in T^{2d}(V)$.
\item If $A\in (\End_\La(V))_{\bar0}^+$ and $X\in(\End_\La(V))_{\bar0}$ is such that
$A=\omega(X)=X^\dag X$, then for any element $\ul v\in T^{2d}(V)$, we have
$H_L(A\ot\dots\ot A\ot\ul v)=F_L(A\ot\ul v)=L(X.\ul v)=L(Xv_1\ot\dots\ot Xv_{2d})$.
\item If $A_i\in (\End_\La(V))_{\bar0}^-$ for some $i$, then for any element $\ul v\in T^{2d}(V)$
and elements $A_i\in\End_\La(V)$,
$H_L(A_1\ot\dots\ot A_d\ot\ul v)=0$.
\end{enumerate}
\end{proposition}

This leads to the following form for the Second fundamental Theorem for $\OSp(V)$.

\begin{theorem}\label{thm:main}
All linear relations among the functions $\kappa_D$ ($D\in\cD$) (see Definition \ref{def:kd}) are
obtained as follows. We have $\sum_{D\in B_{2d}^0}\kappa_D\lambda_D=0$, where $\lambda_D\in\La$ for all $D$,
if and only if we have $\sum_{D\in \cD} h(\kappa_D)\lambda_D=0$ in
$\left((T^{2d}(V^*)\ot T^{2d}(V))^*\right)^{\GL(V)}$.
\end{theorem}

We remark that all relations in $\left((T^{2d}(V^*)\ot T^{2d}(V))^*\right)^{\GL(V)}$ are known from Theorem \ref{thm:sft-gl2}.
Thus in principle, all relations among the $\kappa_D$ in $W^G$ are known. To make effective use of this we shall
prove a more explicit diagrammatic version of Theorem \ref{thm:main} in the next section. This will suffice to describe explicitly
the kernel of the surjective map $\kappa:B_{2d}^0\lr T^{2d}(V^*)^G$ defined by $D\mapsto \kappa_D$.

\section{A diagrammatic form of the second fundamental theorem for $\OSp(V)$}
In this section we continue with the notation of the previous one. Thus $\sdim V_\C=(m|2n)$ and $G=\OSp(V)$.
We shall work in the Brauer category $\CB(m-2n)$, and our standard notation is that for non-negative integers
$s,t$, $B_s^t$ denotes $\Hom_{\CB(m-2n)}(s,t)$, which is the free $\La$-module with basis the set of Brauer diagrams
from $s$ to $t$. 

\subsection{The diagrammatic setup} Consider the following diagram:
\be\label{cd1}
\xymatrix{
 &B_{2d}^0\ar[d]_{\chi?} \ar@{->}[rr]^{\kappa\phantom{XXXX}} &&\left(T^{2d}(V^*)\right)^G\ar[d]^h\\
B_{2d}^{2d} \ar[r]&\Lambda\Sym_{2d}\ar@{->}[rr]_{\delta\phantom{XXXXXX}} && \left(T^{2d}(V^*)\ot T^{2d}(V)\right)^{\GL(V)}		
}
\ee
%
where $\kappa$ and $\delta$ are defined respectively by $\kappa(D)=\kappa_D$ and $\delta(\pi)=\delta_\pi$,
for a Brauer diagram $D$ from $2d$ to $0$ and permutation $\pi\in\La\Sym_{2d}$; $h$ is the map defined in
Proposition \ref{prop:ltohl}, and the map ${B_{2d}^{2d}}\to\La\Sym_{2d}$ is the natural map with kernel the 
ideal spanned by diagrams with horizontal arcs. We shall define a $\La$-linear map $\chi$ as shown, which makes the diagram
commute.

Let $D$ be a diagram from $2d$ to $0$. We shall define $\chi(D)\in \La\Sym_{2d}$.
For the definition we shall require the $C$-symmetriser $e(C)\in\La\Sym_{2d}$, defined by
\be\label{eq:alpha}
e(C):=|C|\inv\alpha^+(C)=|C|\inv\sum_{\sigma\in C}\sigma.
\ee
Now $D=D_0\pi_D$ where  $D_0$ is the diagram defined in Corollary \ref{cor:kd}
and $\pi_D\in\Sym_{2d}$ is unique up to left multiplication by an element $\sigma\in C$.
It follows that the element
\be\label{eq:chid}
\chi(D):=e(C)\pi_D
\ee
is well defined, since for any element $\sigma\in C$, $e(C)\sigma=e(C)$, so that $\chi(D)$ is independent
of the choice of $\pi_D$.

\setlength{\unitlength}{0.30mm}
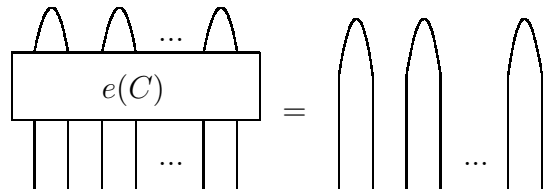
\begin{figure}[h]
\begin{center}
\begin{picture}(280, 100)(0,0)
\qbezier(30, 60)(38, 100)(45, 60)
\qbezier(60, 60)(68, 100)(75, 60)

\put(85, 65){$...$}

\qbezier(105, 60)(113, 100)(120, 60)

\put(20, 60){\line(1, 0){110}}
\put(20, 60){\line(0, -1){30}}
\put(130, 60){\line(0, -1){30}}
\put(20, 30){\line(1, 0){110}}

\put(60, 40){$e(C)$}

\put(30, 30){\line(0, -1){30}}
\put(45, 30){\line(0, -1){30}}
\put(60, 30){\line(0, -1){30}}
\put(75, 30){\line(0, -1){30}}

\put(85, 10){$...$}

\put(105, 30){\line(0, -1){30}}
\put(120, 30){\line(0, -1){30}}

\put(140, 30){$=$}

\qbezier(165, 50)(173, 100)(180, 50)
\qbezier(195, 50)(203, 100)(210, 50)
\qbezier(240, 50)(247, 100)(255, 50)

\put(165, 50){\line(0, -1){50}}
\put(180, 50){\line(0, -1){50}}
\put(195, 50){\line(0, -1){50}}
\put(210, 50){\line(0, -1){50}}

\put(220, 10){$...$}

\put(240, 50){\line(0, -1){50}}
\put(255, 50){\line(0, -1){50}}

\end{picture}
\end{center}
\caption{Composition $D_0e(C)=D_0$}
\label{Fig4}
\end{figure}

\setlength{\unitlength}{0.25mm}
\begin{figure}[h]
\begin{center}
\begin{picture}(200, 120)(10,60)
\qbezier(30, 120)(40, 200)(50, 120)
\qbezier(70, 120)(80, 200)(90, 120)
\qbezier(110, 120)(120, 200)(130, 120)
\put(134, 70){$...$}
\qbezier(160, 120)(170, 200)(180, 120)

\put(18, 120){- - - - - - - - - - - - - - - - }

\put(30, 120){\line(0, -1){60}}

\qbezier(50, 120)(50, 90)(130, 60)

\qbezier(70, 120)(65, 90)(50, 60)

\qbezier(90, 120)(90, 90)(70, 60)

\qbezier(110, 120)(105, 90)(90, 60)

\qbezier(130, 120)(135, 90)(160, 60)

\qbezier(160, 120)(150, 90)(110, 60)

\put(180, 120){\line(0, -1){60}}
\end{picture}
\end{center}
\caption{ Diagram $D$ in Fig. \ref{Fig5} as composition $D_0\pi_D$}
\label{Fig6}
\end{figure}
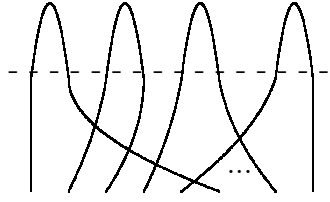

\begin{proposition}\label{prop:chi}
With $\chi:B_{2d}^0\to \La\Sym_{2d}$ defined as in \eqref{eq:chid} above, the diagram \eqref{cd1} commutes.
\end{proposition}
The proof depends on some results to be proved below.

\subsection{Proof of Proposition \ref{prop:chi}}
This will depend on the next two Lemmas.
The first simply identifies the projection $p:T^{2d}(V^*)\lr S^d(S^2(V^*))$.
\begin{lemma}\label{lem:proj}
The canonical projection  $p:T^{2d}(V^*)\lr S^d(S^2(V^*))$ is realised by the action of the idempotent
$\varpi_{2d}(e(C))$, where $e(C)=|C|\inv\sum_{\sigma\in C}\sigma$ is as
defined in \eqref{eq:alpha}.
\end{lemma}

This is clear from the definition of the super symmetric powers.

The next Lemma gives the values of $F_{\kappa_D}$ for $D\in B_{2d}^0$. Recall (see Remark \ref{rem:fl})
that $F_L$ is to be thought of as a linear map from $S^d(S^2(V^*))\ot T^{2d}(V)$ to $\La$, and $H_L$
is its pullback to $T^{2d}(V^*)\ot T^{2d}(V)$ (cf. the last Lemma).
\begin{lemma}\label{lem:fl}
Let $L=\kappa_D$ (see Definition \ref{def:kd}). For $\ul\psi\in S^d(S^2(V^*))\subset T^{2d}(V^*)$ and $\ul v\in T^{2d}(V)$,
we have $F_{\kappa_D}(\ul\psi\ot\ul v)=\delta_{\pi_D}(\ul\psi\ot\ul v)$, where $\pi_D$ is any permutation
in $D$, thought of as a coset $C\pi_D$ of $C$ in $\Sym_{2d}$ and $F_{\kappa_D}$ is defined in \eqref{eq:fl}.
\end{lemma}
\begin{proof} Note that by Lemma \ref{lem:proj}, $S^d(S^2(V^*))=\varpi_{2d}(e(C))T^{2d}(V^*)$.
Since the statement is linear in $\ul\psi$, it follows from Corollary \ref{cor:symspan} above, 
that it suffices to prove the Lemma
for  $\ul{\psi}=A\ot A\ot\dots\ot A\mid A\in (\End_\La(V))_{\bar0}^+ \}\in S^d(S^2(V^*))$,
where $A$ has the form $A=X^\dag X$ for some $X\in (\End_\La(V))_{\bar0}$.
We therefore consider this case.

Then
$$
\begin{aligned}
F_{\kappa_D}(A\ot A\ot\dots\ot A\ot\ul v)&=\kappa_D(Xv_1\ot Xv_2\ot\dots\ot Xv_{2d})\\
&=\langle\varpi_{2d}(\pi_D)(Xv_1\ot\dots\ot Xv_{2d})\rangle\\
&=\mu(A\ot A\ot\dots\ot A\ot\varpi_{2d}(\pi_D)(\ul v)),\\
\end{aligned}
$$
where $\mu$ is the map in Lemma \ref{lem:endinv}, since $(Xv_i,Xv_j)=(v_i,X^\dag Xv_j)=(v_i,Av_j)=(Av_i,v_j)$.

From the commutativity of the diagram in Lemma \ref{lem:endinv}, it follows that
$F_{\kappa_D}(A\ot A\ot\dots\ot A\ot\ul v)=\langle  A\ot A\ot\dots \ot A, \varpi_{2d}(\pi_D)(\ul v)\rangle
=\delta_{\pi_D}( A\ot A\ot\dots \ot A\ot\ul v)$, and replacing $ A\ot A\ot\dots \ot A$ by a general element
$\ul\psi\in S^d(S^2(V^*))$, the result follows.
\end{proof}

\begin{proof}[Proof of Proposition \ref{prop:chi}]
It remains to prove that for each element $\ul\phi\ot\ul v\in T^{2d}(V^*)\ot T^{2d}(V)$,
we have
\be\label{eq:pf}
h(\kappa_D)(\ul\phi\ot\ul v)=\delta(\chi(D))(\ul\phi\ot\ul v).
\ee
By Lemma \ref{lem:proj}, the left side of \eqref{eq:pf} is equal to
$F_{\kappa_D}(p(\ul\phi)\ot\ul v)=F_{\kappa_D}([\varpi_{2d}(e(C))(\ul\phi)]\ot\ul v)$,
which by Lemma \ref{lem:fl} is equal to
$\delta_{\pi_D}([\varpi_{2d}(e(C))(\ul\phi)]\ot\ul v)$.
But
$$
\begin{aligned}
\delta_{\pi_D}([\varpi_{2d}(e(C))(\ul\phi)]\ot\ul v)=&\langle [\varpi_{2d}(e(C))(\ul\phi)],\varpi_{2d}(\pi_D)(\ul v )\rangle\\
=&\langle\ul\phi,\varpi_{2d}(e(C))\varpi_{2d}(\pi_D)(\ul v) \rangle\\
=&\langle\ul\phi,\varpi_{2d}(|C|\inv\sum_{\sigma\in C}\sigma\pi_D)(\ul v) \rangle\\
=&|C|\inv\sum_{\sigma\in C}\delta_{\sigma\pi_D}(\ul\phi\ot\ul v)\\
=&\delta(e(C)\pi_D)(\ul\phi\ot\ul v)\\
=&\delta(\chi(D))(\ul\phi\ot\ul v),
\end{aligned}
$$
and the proof is complete.
\end{proof}

\subsection{The second fundamental theorem of invariant theory} The second fundamental theorem describes the
kernel of the map $\kappa$ in the diagram \eqref{cd1}.
The result, which is the main result of this work, is as follows.
\begin{theorem}\label{thm:Main}
We have
$$
\Ker(\kappa)=D_0 I(m,n),
$$
where $I(m,n)$ is the ideal of $\La\Sym_{2d}$ which is the sum of the two-sided ideals corresponding to partitions which
contain an $(m+1)\times(2n+1)$ rectangle, and $D_0$ is the diagram defined in Corollary \ref{cor:kd}.
\end{theorem}
\begin{proof}
Since $h$ is injective, it follows that $\Ker(\kappa)=\Ker (\delta\circ\chi)$. Hence $b\in B_{2d}^0$ lies in
$\Ker (\kappa)$ if and only if $\chi(b)\in\Ker(\delta)$, and by Theorem \ref{thm:sft-gl2}, $\Ker(\delta)=I(m,n)$.
But we claim that $\chi\inv(I(m,n))=D_0I(m,n)$, for if $x\in I(m,n)$ then $\chi(D_0x)=e(C)x\in I(m,n)$, while conversely
if $D_0x\in \chi\inv(I(m,n))$, then $\chi(D_0x)=e(C)x\in I(m,n)$. But then $D_0x=D_0e(C)x\in D_0I(m,n)$,
proving the assertion. The Theorem follows.
\end{proof}

\setlength{\unitlength}{0.3mm}
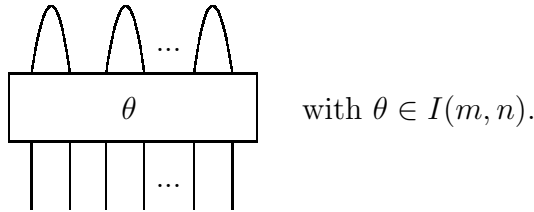
\begin{figure}[h]
\begin{center}
\begin{picture}(200, 100)(0,0)
\qbezier(30, 60)(39, 120)(47, 60)
\qbezier(63, 60)(72, 120)(80, 60)
\put(85, 70){$...$}
\qbezier(102, 60)(110, 120)(119, 60)

\put(20, 60){\line(1, 0){110}}
\put(20, 60){\line(0, -1){30}}
\put(130, 60){\line(0, -1){30}}
\put(20, 30){\line(1, 0){110}}

\put(70, 40){$\theta$}

\put(30, 30){\line(0, -1){30}}
\put(47, 30){\line(0, -1){30}}
\put(63, 30){\line(0, -1){30}}
\put(80, 30){\line(0, -1){30}}
\put(85, 10){$...$}
\put(102, 30){\line(0, -1){30}}
\put(119, 30){\line(0, -1){30}}

\put(150, 40){with $\theta\in I(m, n)$.}

\end{picture}
\end{center}
\caption{Elements of $\Ker(\kappa)$}
\label{Fig7}
\end{figure}

The following statement concerns the range of values of $d$ for which
\begin{corollary}\label{cor:lowest} With notation as in Theorem \ref{thm:Main}, the kernel of $\kappa$ is zero if
$2d<(m+1)(2n+1)$. Further $\Ker(\kappa)\neq 0$ for $d\geq (m+1)(2n+1)$.
\end{corollary}
\begin{proof}
The first assertion is immediate from Theorem \ref{thm:Main}. The second follows from the fact that
$I(m,n)=\bU^{(m+1)(2n+1)}(D\circ\left(I(m,n)\ot(1^{\ot (m+1)(2n+1)}\right))$, where $D:2(m+1)(2n+1)\to 0$
is the diagram with arcs $(i,2(m+1)(2n+1)-i+1)$, $i=1,\dots,(m+1)(2n+1)$ and $\bU$ is the map described in the first paragraph
of \S \ref{ss:ff} below.  The argument of $\bU^{(m+1)(2n+1)}$ on the right side is therefore a
non-zero element of $\Ker(\kappa)$.
\end{proof}

Note that as the discussion of the classical orthogonal and symplectic cases below shows, Corollary \ref{cor:lowest} is not optimal.
Here is a non-classical example.
\begin{example}\label{ex:osp12}
We take $m=1=n$; thus we are in the case of $\OSp(1|2)$.
In this case the smallest value of $d$ for which $\Ker(\kappa)$ could be non-zero is $d=3$ ($2d=(m+1)(2n+1)=2\times 3=6$).
However a straightforward calculation shows that in this case (i.e. $d=3$) we have $D_0I(1,2)=0$ in $B_6^0$. Thus
here $\Ker (\kappa)=0$ when $2d=(m+1)(2n+1)$. In fact in this case one sees that $\Ker(\kappa)=0$ if and only if
$d\leq 3$. For this one must verify that $\Ker(\kappa)\neq 0$ when $d=4$, i.e. that $D_0I(1,2)\neq 0$ in $B_8^0$.
\end{example}

\begin{remark}
The available evidence seems to point to the guess that the smallest value of $d$ for which $\Ker(\kappa)\neq 0$ is 
$d=(m+1)(n+1)$ (i.e. $2d=(m+1)(2n+1)+m+1$). However we have not verified this statement, which of course is true
in the classical cases $m=0$ or $n=0$.
\end{remark}
\subsection{Interpretation in the first formulation}\label{ss:ff} Theorem \ref{thm:Main} provides the SFT for $\OSp(V)$ in its
second formulation; that is, it describes the linear structure of $(T^{2d}(V))^{\OSp(V)}$.
 It may be useful to point out that it may be reinterpreted in terms of the Brauer algebra action
$\eta_r:B_r(m-2n)\lr \End_{\OSp(V)}(T^r(V))$ (see \cite[Cor. 5.7]{LZ6}).
For this we note that in the Brauer category \cite{LZ6}, for non-negative integers $p,q$ with $p>0$
there is an isomorphism $\bU:B_p^q\overset{\sim}{\lr} B_{p-1}^{q+1}$. It is denoted $\bU_{p-1}^1$
in \cite[Cor. 2.8]{LZ5}, and is depicted in Figure \ref{Fig9}.

\setlength{\unitlength}{0.3mm}
\begin{figure}[h]
\begin{center}
\begin{picture}(200, 100)(0,0)
\put(50, 70){$\dots q-1\dots$}

\put(0, 30){\line(0, 1){60}}

\put(30, 60){\line(0, 1){30}}
\put(45, 60){\line(0, 1){30}}
\put(119, 60){\line(0, 1){30}}

\put(20, 60){\line(1, 0){110}}
\put(20, 60){\line(0, -1){30}}
\put(130, 60){\line(0, -1){30}}
\put(20, 30){\line(1, 0){110}}

\put(70, 40){$D$}
\qbezier(0, 30)(18, -30)(30,30)
\put(45, 30){\line(0, -1){30}}

\put(50, 10){$\dots p-1\dots$}
\put(119, 30){\line(0, -1){30}}


\end{picture}
\end{center}
\caption{$\bU(D)$ for $D\in B_p^q$}
\label{Fig9}
\end{figure}
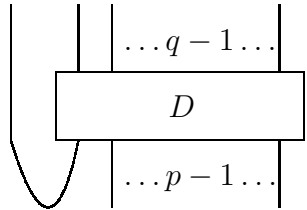

One therefore has an isomorphism $\bU^d:B_{2d}^0\lr B_d^d$, and this leads to the following extension of the diagram
\eqref{cd1}.

\be\label{cd2}
\begin{CD}
B_d^d @>\eta>> \End_G(T^d(V))\\
@A\bU^dA\cong A@ VV\cong V\\
{B^{0}_{2d}} @>\kappa>> {\left(T^{2d}(V^*)\right)^G} \\
@ V\chi VV @VVhV \\
\La\Sym_{2d}\subseteq{B_{2d}^{2d}} @>>\delta> {\left(T^{2d}(V^*)\ot T^{2d}(V)\right)^{\GL(V)}.}
\end{CD}
\ee

The top map $\eta:B_d^d(m-2n)\to\End_G(T^d(V))$ is precisely the map discussed in the paper \cite{LZ5}.
Theorem \ref{thm:Main} has the following evident consequence.

\begin{corollary}\label{cor:Main}
The kernel of the surjective homomorphism $\eta:B_d(m-2n):=B_d^d\to\End_G(T^d(V))$ is $\bU^d(D_0I(m,n))$.
\end{corollary}

\begin{remark}
In \cite{LZ4} and \cite{LZ5} we showed that in the classical cases when $m=0$ or when $n=0$, the kernel is actually generated by a single idempotent
in $B_d(m-2n)$, but we do not yet have a similar result in the general case.
\end{remark}
\section{Application to the classical groups}

In this final section we shall show how our main theorem provides
a new proof of the second fundamental theorem in the classical cases of the
orthogonal and symplectic groups over $\C$  (see, e.g., \cite{GW, P2}). The principal observation is that Theorem \ref{thm:Main}
and its proof remain valid {\it mutatis mutandis} in these cases when $\La$ is replaced by $\C$,
$V=V_\C\ot_\C\La$ by $V_\C$, and $\OSp(V_\C\ot_\C\La)$
is replaced by $G=O(m,\C)$ or $\Sp(2n,\C)$.

In this section we therefore take $V$ to be $V_\C$, and apply Theorem \ref{thm:Main} in those respective cases.

\subsection{The othogonal case: $G=O(m,\C)$} Here we take $V_\C=(V_\C)_{\bar0}$ to be purely even, and the form
$(-,-)$ to be symmetric. Then parts (1)-(3) of Proposition \ref{prop:ltohl} recover the key elements in the proof
of the first fundamental theorem of the invariant theory of $O(m, \C)$ given in \cite{ABP}.
The ideal $I(m,n)=I(m,0)$ is spanned (over $\C$) by elements of the form
$\pi\alpha^-(m+1)\pi'$, where $\pi,\pi'\in\Sym_{2d}$ and for $\ell$ such that $1\leq\ell\leq 2d$,  $\alpha^-(\ell)=\sum_{\sigma\in\Sym_{\ell}}\varepsilon(\sigma)\sigma$,
where $\ve$ is the alternating character of $\Sym_{2d}$ and $\Sym_{\ell}\subseteq\Sym_{2d}$ is the subgroup which permutes the
first $\ell$ symbols.

Theorem \ref{thm:Main} asserts that $\Ker\left(\kappa: B_{2d}^0(\C)\lr (T^{2d}(V^*)^G\right)$ is spanned by the elements $D_0\pi\alpha^-(m+1)\pi'$.
The next Lemma shows that if $d$ is small, the kernel is zero.
\begin{lemma}\label{lem:smalld}
If $d\leq m$, then each element $D_0\pi\alpha^-(m+1)\pi'$ of $B_{2d}^0$ is zero.
\end{lemma}
\begin{proof}
Let $\pi\in\Sym_{2d}$ and consider the diagram $D=D_0\pi$.
If $d\leq m$, then $d<m+1$, and it follows that at least one arc of $D$ has both ends in $\{1,2,\dots,m+1\}$
(if each arc with an end in $\{1,2,\dots,m+1\}$ had an end outside $\{1,2,\dots,m+1\}$, we would have $2d\geq 2(m+1)$).
Suppose this arc is from $i$ to $j$, where $1\leq i<j\leq m+1$, and write $(ij)$ for the transposition in $\Sym_{2d}$ which
interchanges $i$ and $j$.

Then by the rules for multiplying Brauer diagrams, we have $D(ij)=D$, whence $D_0\pi(ij)\alpha^-(m+1)
=D_0\pi\alpha^-(m+1)$.

But by the alternating property of $\alpha^-(m+1)$, we have $(ij)\alpha^-(m+1)=-\alpha^-(m+1)$,
whence $D_0\pi(ij)\alpha^-(m+1)=-D_0\pi\alpha^-(m+1)$. The Lemma follows.
\end{proof}

\setlength{\unitlength}{0.25mm}
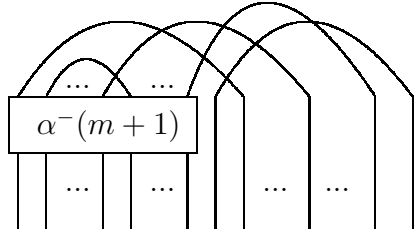
\begin{figure}[h]
\begin{center}
\begin{picture}(250, 120)(10,-10)
\qbezier(30, 60)(80, 140)(150, 60)
\qbezier(45, 60)(65, 100)(90, 60)
\qbezier(75, 60)(120, 140)(185, 60)
\qbezier(120, 60)(155, 160)(220, 60)
\qbezier(135, 60)(180, 140)(240, 60)

\put(100, 65){$...$}
\put(55, 65){$...$}

\put(25, 60){\line(1, 0){100}}
\put(25, 60){\line(0, -1){30}}
\put(125, 60){\line(0, -1){30}}
\put(25, 30){\line(1, 0){100}}

\put(40, 40){$\alpha^-(m+1)$}

\put(30, 30){\line(0, -1){40}}
\put(45, 30){\line(0, -1){40}}
\put(75, 30){\line(0, -1){40}}
\put(90, 30){\line(0, -1){40}}
\put(120, 30){\line(0, -1){40}}

\put(135, 60){\line(0, -1){70}}
\put(185, 60){\line(0, -1){70}}
\put(220, 60){\line(0, -1){70}}
\put(150, 60){\line(0, -1){70}}
\put(240, 60){\line(0, -1){70}}

\put(55, 10){$...$}
\put(100, 10){$...$}
\put(160, 10){$...$}
\put(193, 10){$...$}
\end{picture}
\end{center}
\caption{A zero element in $B_{2d}^0$}
\label{Fig8}
\end{figure}

\begin{corollary}\label{cor:smalld}
If $d\leq m$, $\Ker(\kappa)=0$, and $\kappa:B_{2d}^0(\C)\lr (T^{2d}(V^*))^G$ is an isomorphism.
\end{corollary}

Our main theorem now has the following interpretation in the present case (cf. \cite[Theorem 3.4]{LZ5}).
Recall that diagrams $D\in B_{2d}^0$ are in canonical bijection with partitions of $\{1,\dots,2d\}$
into pairs.

\begin{theorem}\label{thm:orth} Let $V=V_{\bar0}=\C^m$ and $G=O(n,\C)$.
If $d\leq m$, then $\Ker(\kappa)=0$. Assume that $d\geq m+1$.

Define a set of linear relations among the $\kappa_D$ as follows. Let $S,S'$ be two disjoint
subsets of $\{1,\dots,2d\}$ such that $|S|=|S'|=m+1$,
 and for $\sigma\in\Sym(\{S\})$, let $D(\pi,\pi';\sigma)$ be the diagram in $B_{2d}^0$ which pairs
$\sigma(s_{i})$ with $s'_i$, where the $s_i\in S,s'_i\in S'$ are written in increasing order, and the points in
$\{1,\dots,2d\}\setminus(S\cup S')$ are paired as they
are in each diagram occurring in $D_0\pi\alpha^-(m+1)\pi'$.

Then for each $\pi,\pi'\in\Sym_{2d}$,
\be\label{eq:relo}
\sum_{\sigma\in\Sym_{m+1}}\ve(\sigma)\kappa_{D(\pi,\pi',\sigma)}=0,
\ee
and all relations among the $\kappa_D$ are linear consequences of these relations.
\end{theorem}

\begin{proof} The first statement is Corollary \ref{cor:smalld}.

By Theorem \ref{thm:Main}, all relations among the $\kappa_D$
are consequences of the relations $\kappa(D_0\pi\alpha^-(m+1)\pi')=0$, for $\pi,\pi'\in \Sym_{2d}$.

Now the argument of the proof of Lemma \ref{lem:smalld} shows that if two of the points in
the set $S:= \{\pi\inv(1),\pi\inv(2),\dots,\pi\inv(m+1)\}$ are paired by $D_0$ then $D_0\pi\alpha^-(m+1)\pi'=0$
in $B_{2d}^0$, and so to obtain a non-zero element of the kernel, we may assume that no two points in $S$ are paired
by $D_0$, and hence that the set $S'$ of points which are paired by $D_0$ with a point of $S$ is disjoint from $S$.

In this case, the relation $\kappa(D_0\pi\alpha^-(m+1)\pi')=0$ is easily seen to translate into \eqref{eq:relo}, and
the result is now clear.
\end{proof}

\setlength{\unitlength}{0.30mm}
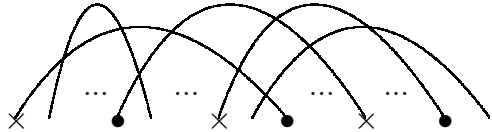
\begin{figure}[h]
\begin{center}
\begin{picture}(250, 60)(10,60)
\qbezier(30, 60)(80, 140)(150, 60) 	\put(25, 55){$\times$}\put(147, 55){$\bullet$}
\qbezier(45, 60)(65, 160)(90, 60)		
\qbezier(75, 60)(120, 160)(185, 60)	\put(72, 55){$\bullet$} \put(180, 55){$\times$}
\qbezier(120, 60)(155, 160)(220, 60)	\put(217, 55){$\bullet$}\put(115, 55){$\times$}
\qbezier(135, 60)(180, 140)(240, 60)	
\put(60, 70){$...$}
\put(100, 70){$...$}
\put(160, 70){$...$}
\put(193, 70){$...$}
\end{picture}
\end{center}
\caption{The set $S$ consists of $\bullet$'s and $S'$ of $\times$'s}
\label{Fig10}
\end{figure}

\subsection{The symplectic case} In this case we take $V=V_\C=V_{\bar1}$. The form  $(-,-)$ is then skew, and $G=\Sp(2n,\C)$.
The ideal $I(m,n)=I(0,n)$ is in this case spanned (over $\C$) by elements of the form
$\pi\alpha^+(2n+1)\pi'$, where $\pi,\pi'\in\Sym_{2d}$ and for each integer $\ell$ with $1\leq\ell\leq 2d$, $\alpha^+(\ell)=\sum_{\sigma\in\Sym_{\ell}}\sigma$,
where as above, $\Sym_{\ell}\subseteq\Sym_{2d}$ is the subgroup which permutes the
first $\ell$ symbols.

In this case, the analogue of Lemma \ref{lem:smalld} is the trivial observation that there are no non-zero elements of the kernel unless $2d\geq 2n+1$, i.e. $d\geq n+1$.

Taking into account that for a symplectic form $(v,v)=0$ for all $v\in V$, we obtain the following form of the main theorem for the symplectic case.

\begin{theorem}\label{thm:symp}
Let $V=V_{\bar1}=\C^{2n}$, and $G=\Sp(2n,\C)$. If $d\leq n$, then $\Ker(\kappa)=0$. Assume that $d\geq n+1$.

Define a set of linear relations among the $\kappa_D$ as follows. Let $S$ be a
subset of $\{1,\dots,2d\}$ such that $|S|=2n+1$. For any diagram $D\in B_{2d}^0$ and $\sigma\in\Sym(\{S\})$, let $D(S,\sigma)$ be
 the diagram $D(S,\sigma)=D\sigma$,
where $\sigma$ is regarded as an element of $\Sym_{2d}\subset B_{2d}^{2d}$ which fixes the points outside $S$.

Then
\be\label{eq:relsp}
\sum_{\sigma\in\Sym(\{S\})}\kappa_{D(S,\sigma)}=0,
\ee
and all relations among the $\kappa_D$ are linear consequences of these relations.
\end{theorem}

This is clear from Theorem \ref{thm:Main}.



\end{document}